\documentclass[a4paper,11pt]{article}
\usepackage{amsmath,amsthm,amssymb,fullpage}

	\newcommand{\R}{{\mathbb R}}

	\newcommand{\RR}{{\mathbb R}}
	\newcommand{\vphi}{\varphi}
	\newcommand{\eps}{\varepsilon}
	
	\newtheorem{theo}{Theorem}
	
	\newtheorem{lemma}[theo]{Lemma}
	\newtheorem{proposition}[theo]{Proposition}
	\newtheorem{corollary}[theo]{Corollary}
	
	\theoremstyle{remark}
	\newtheorem*{rem}{Remark}

\begin{document}
\title{Spectral gaps, symmetries and log-concave perturbations}
\author{Franck Barthe and Bo'az Klartag}
\date{}
\maketitle

\abstract{We discuss situations where perturbing a probability measure on $\RR^n$ does not deteriorate its Poincar\'e constant by  much.
A particular example is the symmetric exponential measure in $\RR^n$,  even log-concave perturbations of which  have Poincar\'e constants that grow at most  logarithmically with  the dimension.
This leads to estimates for the Poincar\'e constants of $(n/2)$-dimensional sections of the unit ball of $\ell_p^n$ for $1 \leq p \leq 2$, which are optimal up to logarithmic factors.
 We also consider symmetry properties of the eigenspace of the  Laplace-type operator associated with a log-concave measure. Under symmetry assumptions we show that the dimension of this space is exactly $n$, and we exhibit a certain interlacing between the ``odd'' and ``even'' parts of the spectrum.
}

\section{Introduction}
This work was partly motivated by the study 
 of a family of probability measures on $\mathbb R^n$ which
naturally appear when considering statistical questions pertaining to sparse linear modeling:
\[ d\nu^{n,Q}(x)=\frac1Z e^{-\|x\|_1-Q(x)}\, dx, \]
where $Q$ is a nonnegative quadratic form,  $\|x\|_p=(\sum_i |x_i|^p)^{1/p}$, and $Z=Z_{n,Q}$ is the normalizing constant so that $\nu^{n,Q}$ is a probability measure. The latter is related to the classical functional $\theta\mapsto \|y-X\theta\|_2^2+\lambda \|\theta\|_1$
that one minimizes in order to find the LASSO estimator, see e.g. \cite{bvdgBOOK}. Here the quadratic
term is supposed to ensure a good fit to data $y$, while minimizing the $L_1$ norm favours a small support for the estimator $\theta$.

For a probability measure $\mu$ on $\mathbb R^n$, we denote by $C_P(\mu)$ the Poincar\'e constant of $\mu$, that is the least constant $C$ such that the following inequality holds for all locally Lipschitz functions $f:\mathbb R^n\to \mathbb R$:
\begin{equation} \mathrm{Var}_{\mu} (f)\le C \int_{\mathbb R^n} |\nabla f|^2 d\mu. \label{eq_106} \end{equation}
Here $\mathrm{Var}_{\mu} (f)=\int (f-\int f\, d\mu)^2 d\mu$ if $f\in L_2(\mu)$, and $+\infty$ otherwize, denotes the variance of $f$ with respect to $\mu$. Such Poincar\'e inequalities, when they hold, allow to quantify concentration properties of $\mu$ as well as relaxation properties of associated Langevin dynamics, see e.g. \cite{BGLbook}.
\medskip

 A natural question, posed to us by S. Gadat, is whether the Poincar\'e constant of $\nu^{n,Q}$ can be upper bounded independently of the quadratic form $Q$. This seems plausible, as the addition of $Q$ only makes the measure more log-concave and more localized around the origin. But making this intuition rigorous is far from obvious.
A more demanding question is whether $C_P(\nu^{n,Q})$ is maximal when $Q=0$.
Observe that  $\nu^{n,0}=\nu^n$ is the $n$-fold product of the Laplace distribution on $\mathbb R$, $d\nu(t)=\exp(-|t|)\, dt/2$. By the tensorization property of Poincar\'e inequalities, we have $C_P(\nu^n)=C_P(\nu)= 4$ (see Lemma 2.1 in \cite{bobkov-ledoux} for
$C_P(\nu)\le 4$, the converse inequality is checked with exponential test functions).  A positive answer to the latter question would imply that $C_P(\nu^{n,Q})$ is upper bounded by 4, independently of the dimension and of the nonnegative  quadratic form $Q$.
We cannot establish this bound, but we provide results in this direction which apply to more general settings, while putting forward the relevent features of the problem as symmetry, log-concavity and appropriate comparison with the Gaussian case. A sample result is stated next:

\begin{theo}\label{th:Exp-poincare}
	Let $n \geq 2$ and let $F: \RR^n \rightarrow \RR$ be an even, convex function and let $1 \leq p \leq 2$.
	Consider the probability measure $\mu$  
	 on $\RR^n$ given by
	$$ d \mu(x) = \frac{1}{Z} e^{-\| x \|_p^p - F(x)}  dx $$
	where $Z$ is a normalizing constant.
	Then,
	\[ C_P(\mu) \le C (\log n)^{\frac{2-p}{p}},\]
	where $C$ is a universal constant.
\end{theo}

We do not know whether the logarithmic factor in
Theorem \ref{th:Exp-poincare} is necessary. Up to this logarithmic factor,
this theorem provides a positive answer to the above question, since a non-negative quadratic function $Q$ is an even, convex function. Note that in the case where $p=2$, there is no logarithmic factor in Theorem \ref{th:Exp-poincare}, yet in this case the Theorem is well-known and it holds true without  the assumption that $F$ is an even function  (see Corollary \ref{cor:BL} below).
The case where $p\in [1,2)$ is harder, and
relies on techniques from the study of log-concave measures. Using a result of Kolesnikov-Milman \cite{kolesnikov-milman} 
that allows to compare Poincar\'e constants of log-concave functions and their level sets, we obtain the following:
\begin{corollary} Let $n \geq 2$ and $p \in [1,2]$.
	Let $E \subseteq \RR^n$ be a linear subspace, and set $\kappa = \dim(E) / n$. Then,
	\[ C_P\big(\lambda_{B_p^n\cap E}\big) \le c(\kappa) \cdot \log^{\frac{2}{p}} (n) \cdot \sup_{0 \neq \theta \in \RR^n} \int_{ B_p^n\cap E} \left \langle x, \frac{\theta}{|\theta|} \right \rangle^2 d \lambda_{B_p^n\cap E}(x),\]
	where $c(\kappa)$ depends solely on $\kappa \in [0,1]$, where $B_p^n = \{ x \in \RR^n \, ; \, \sum_i |x_i|^p \leq 1 \}$, and where $\lambda_{B_p^n \cap E}$ is the uniform probability measure on the section $B_p^n \cap E$. \label{cor1}
\end{corollary}
This provides a partial confirmation, up to a logarithmic term, of a famous conjecture of Kannan, Lov\'asz and Simonovits, which we recall
in Section \ref{sec:KLS}.

\medskip

 In Section \ref{sec_GM} we
present additional related results, and in particular a slightly more general version of the above results, see Theorem \ref{th:GM-poincare}.
 The proofs in Section \ref{sec_GM} rely on ideas from the recent Gaussian-mixtures analysis of Eskenazis, Nayar and Tkocz \cite{ENT-mixtures},
and on the fact going back to \cite{klartag-unconditional}, that the first non-trivial eigenfunction is an odd function under convexity and symmetry assumptions. This fact is revisited here, and in particular we prove the following  interlacing result for the spectrum of the Laplace-type operator associated with an even, log-concave measure. A function $f: \RR^n \rightarrow [0, \infty)$ is log-concave if the set where it is positive is convex, and $-\log f$ is a convex function on this set.

\begin{theo}\label{th:interlace}
	Let $\mu$ be a finite measure with a log-concave density in $\RR^n$. Assume that $\mu$ is even. Then  in the definition  (\ref{eq_106}) it suffices to consider odd functions, i.e., denoting $\lambda_P(\mu) = 1 / C_P(\mu)$ we have
	$$ \lambda_P(\mu) = \lambda_P(\mu, ``odd") := \inf_{f: \RR^n \rightarrow \RR \textrm{ is odd}} \frac{\int_{\RR^n} |\nabla f|^2 d \mu}{\mathrm{Var}_{\mu}(f)}, $$
	where the infimum runs over all locally-Lipschitz, odd functions $f \in L^2(\mu)$ with $f \not \equiv 0$.
	
	\medskip
	Moreover, the even functions do not lag too far behind in the spectrum. Specifically, for any $(n+1)$-dimensional subspace $E \subseteq L^2(\mu)$ of locally-Lipschitz, odd functions we have
	$$ \lambda_P(\mu, ``even") := \inf_{f: \RR^n \rightarrow \RR \textrm{ is even}} \frac{\int_{\RR^n} |\nabla f|^2 d \mu}{\mathrm{Var}_{\mu}(f)}
	\leq \sup_{0 \not \equiv f \in E}
	\frac{\int_{\RR^n} |\nabla f|^2 d \mu}{\mathrm{Var}_{\mu}(f)}, $$
	where the infimum runs over all locally-Lipschitz, even functions $f \in L^2(\mu)$ with $f \not \equiv Const$.
\end{theo}

It is well-known that there exists log-concave measures, such as the Laplace distribution mentioned above, for which the infimum defining the Poincar\'e constant is {\it not} attained. Nevertheless, under mild regularity assumptions on $\mu$ it is known that an eigenspace $E_{\mu} \subseteq L^2(\mu)$ corresponding to the eigenvalue $\lambda_P(\mu)$ does exist, and by elliptic regularity the eigenfunctions are smooth. The eigenspace $E_{\mu}$ consists of all locally-Lipschitz functions $f \in L^2(\mu)$ with $\int f d \mu = 0$ for which
$$ \int_{\RR^n} f^2 d \mu = C_P(\mu) \cdot \int_{\RR^n} |\nabla f|^2 d \mu. $$
Given a measure $\mu$ on $\RR^n$ write $\mathcal O_n(\mu)$ for the group
of all linear isometries $R: \RR^n \rightarrow \RR^n$ with $R_* \mu = \mu$.
As an example, if $\mu$ has the symmetries of the cube  $[-1,1]^n$, then the group $\mathcal O_n(\mu)$ has at least $2^n \cdot n!$ elements, and it has no non-trivial invariant subspaces.

\begin{theo} Let $\mu$ be a log-concave probability measure on $\RR^n$ with $E_{\mu} \neq \{ 0 \}$. Assume that the group $\mathcal O_n(\mu)$ has no non-trivial invariant subspace in $\RR^n$. Moreover we make the regularity assumption that $\mu$ has a $C^2$-smooth, positive density $e^{-\psi}$
	and that the Hessian matrix of $\psi$ is non-singular at any point of $\RR^n$. Then $$ \dim E_\mu=n. $$
	Moreover, for any $f\in E_\mu\setminus\{0\}$,
	\[E_\mu=\mathrm{span}\big\{ f\circ R; \, R\in \mathcal O_n(\mu)\big\}.\]
	\label{theo_427_}
\end{theo}

The proofs of the last two results appear in Section \ref{sec_pc}, where an extended discussion and several other related results may be found.

\medskip
{\it Acknowledgement.} This paper is  based upon work supported by the National Science Foundation under Grant No. DMS-1440140 while two of the authors were in residence at the Mathematical Sciences Research Institute in Berkeley, California, during the Fall 2017 semester.

\section{Poincar\'e constants for log-concave measures}
\label{sec_pc}

\subsection{Perturbation principles}

We collect here several useful results on the Poincar\'e constants, dealing with various kinds of perturbations of measures.
We start with recalling the classical bounded perturbation principle. It  follows from the  representation formula $\mathrm{Var}_\mu(f)=\inf_{a\in \RR} \int (f-a)^2 d\mu$.
\begin{proposition}\label{prop:bounded-perturbation}
Let $\mu$ be a probability measures on $\mathbb R^n$ and let
$\nu(dx)=e^{V(x)} \mu(dx)$ be another probability measure.
If the function $V$ is bounded, then
\[ C_P(\nu)\le C_P(\mu)\,  e^{\mathrm{Osc}(V)},\]
where $\mathrm{Osc}(V)=\sup V-\inf V$ is called the oscillation of $V$.
\end{proposition}

Denote $\RR_+ = [0, \infty)$. The following one-dimensional comparison
result  appears in \cite{roustant-b-i}:

\begin{proposition}\label{prop:unimodal-perturbation}
Let $b\in(0,\infty]$ and $V$ be an even continuous function on $\mathbb R$ such that  $\mu(dx)=\mathbf{1}_{(-b,b)}(x) e^{-V(x)} dx$ is a  probability measure on $\mathbb R$. Let $\rho:\mathbb R\to \mathbb R^+$ be an even function which is non-increasing on $\mathbb R^+$, such that $\nu(dx)=\rho(x)\, \mu(dx)$ is a probability measure. Then $C_P(\nu)\le C_P(\mu)$.
\end{proposition}
The next statement is known as the Brascamp-Lieb variance inequality. A similar result in the complex setting appeared earlier in H\"ormander's work.
\begin{theo}[Brascamp-Lieb \cite{brascamp-Lieb}]
Let $V:\mathbb R^n\to \mathbb R$ be a $C^2$ function such that for all $x\in \mathbb R^n$, the Hessian  matrix $D^2V(x)$ is positive definite. If $\mu(dx):=e^{-V(x)} dx$ is a probability measure, then for all locally Lipschitz functions $f:\mathbb R^n\to \mathbb R$,
\[ \mathrm{Var}_\mu(f)\le \int \big\langle (D^2V)^{-1}\nabla f, \nabla f\big\rangle \, d\mu.\]
\end{theo}
In particular, if  $D^2V(x)\ge \Sigma^{-1}$ for all $x \in \RR^n$, where $\Sigma$ is a fixed positive-definite matrix, then  for all $f$, $$ \mathrm{Var}_\mu(f)\le \int \big\langle \Sigma\nabla f, \nabla f\big\rangle \, d\mu. $$ Observe that $D^2V(x)\ge \Sigma^{-1}$
means that $x\mapsto V(x)-\frac12 \langle \Sigma^{-1}x,x\rangle$ is convex. This leads, by approximation  (or via a different proof, as in  \cite{bobkov-ledoux-BLBM} where a stronger log-Sobolev inequality is proved), to the following estimate for log-concave perturbations of Gaussian measures.
\begin{corollary}\label{cor:BL}
Let $\Sigma$ be a symmetric, positive-definite $n \times n$ matrix. Let $\rho: \mathbb R^n\to \mathbb R^+$ be a log-concave function, such that
$\mu(dx):=\rho(x) \exp(-\frac 12 \langle \Sigma^{-1} x, x \rangle) dx$ is a probability measure. Then for all locally Lipschitz functions $f:\mathbb R^n\to \mathbb R$:
\[ \mathrm{Var}_\mu(f)\le \int \big\langle \Sigma\nabla f, \nabla f\big\rangle \, d\mu.\]
\end{corollary}
In the log-concave case, Proposition \ref{prop:bounded-perturbation} may be improved substantially, as shown by
E. Milman. A probability measure in $\RR^n$ is log-concave if it is supported in an affine subspace, and admits a log-concave density in this subspace. The total variation distance between two probability measures $\mu$ and $\nu$ is
$$ d_{TV}(\mu, \nu) = \sup_{A} |\mu(A) - \nu(A)| $$
where the supremum runs over all measurable sets $A$.

\begin{theo}[E. Milman, Section 5 in \cite{emanuel}]\label{th:total-variation}
Let $\mu_1$ and $\mu_2$ be two log-concave probability measures on $\mathbb R^n$ and let $\eps > 0$. If $d_{TV}(\mu_1,\mu_2)\le 1-\varepsilon$, then
\[ C_P(\mu_2)\le c(\varepsilon) \cdot C_P(\mu_1),\]
where $c(\varepsilon)$ depends only on $\varepsilon$.
\end{theo}

\subsection{Background on  the KLS conjecture}\label{sec:KLS}
In the seminal paper \cite{kls}, Kannan, Lov\'asz and Simonovits (KLS for short) formulated a conjecture on the Cheeger isoperimetric inequality for convex sets, which turned out to be of fundamental importance for the understanding of volumetric properties of high dimensional convex bodies. We refer to the books \cite{abBOOK,BOOKgreek} for an extensive presentation of the topic, and focus on the material that is needed for the present work.
The KLS conjecture has several equivalent formulations. The one that fits to our purposes is expressed in spectral terms. For a probability measure $\mu$ on $\mathbb R^n$ with finite second moments, let
$C_P(\mu, ``linear")$ denote the least number $C$ such that for every \emph{linear} function $f:\mathbb R^n\to \mathbb R$ it holds $\mathrm{Var}_\mu(f)\le C \int |\nabla f|^2 d\mu$. Plainly
\[ C_P(\mu) \le C_P(\mu, ``linear")= \|\mathrm{Cov}(\mu) \|_{op}.\]
Here, $\mathrm{Cov}(\mu) = (C_{ij})_{i,j=1,\ldots,n}$ is the covariance matrix of $\mu$, with entries
$$ C_{ij} = \int_{\RR^n} x_i x_j d \mu(x) -  \int_{\RR^n} x_i d \mu(x)
\int_{\RR^n} x_j d \mu(x),$$
and  $\|\mathrm{Cov}(\mu) \|_{op}$ is norm of $\mathrm{Cov}(\mu)$ considered as on operator on the Euclidean space $\mathbb R^n$, which is equal to the largest eigenvalue of
$\mathrm{Cov}(\mu)$.

The KLS conjecture predicts the existence of a universal constant $\kappa$ such that for every dimension $n$ and for every compact convex $K\subset \mathbb R^n$ with non-empty interior (convex body),
\[ C_P(\lambda_K)\le \kappa \, C_P(\lambda_K,``linear"), \]
where $\lambda_K$ denotes the uniform probability measure on $K$.
The conjecture has been verified for only a few families of convex bodies as the unit balls of $\ell_p^n$ \cite{sodin,latalaw}, simplices \cite{barthe-wolff}, bodies of revolution \cite{huet},
some Orlicz balls \cite{kolesnikov-milman}. The second named author proved in \cite{klartag-unconditional} that
\[C_P(\lambda_K)\le c \log(1+n)^2 C_P(\lambda_K,``linear"),\]
with $c$ being a universal constant, holds for all convex bodies $K\subset \mathbb R^n$
which are invariant by all coordinate changes of signs ($(x_1,\ldots,x_n)\in K \Longleftrightarrow
 (|x_1|,\ldots,|x_n|)$). Such bodies are called unconditional. See \cite{barthe-cordero} for more general symmetries. Corollary \ref{cor1} above gives another instance of a weak confirmation of the conjecture up to logarithms.

 The KLS conjecture can be formulated in the wider setting of log-concave probability measures
 (it turns out to be equivalent to the initial formulation on convex bodies). Let $\kappa_n$
 denote the least number such that
 \[ C_P(\mu)\le \kappa_n \, C_P(\mu,``linear") \]
 holds for all log-concave probability measures on $\mathbb R^n$. With this notation the KLS conjecture predicts that  $\sup_k \kappa_n<+\infty$. We will use known estimates on $\kappa_n$. A rather easy bound was given by Bobkov \cite{bobkov}, extending the original result of \cite{kls} for convex bodies: for all log-concave probability measures on $\mathbb R^n$,
 \begin{equation}\label{eq:trace-bound}
  C_P(\mu) \le c\, \mathrm{Tr}(\mathrm{Cov}(\mu)),
  \end{equation}
 where $c$ is a universal constant. This gives $\kappa_n\le c\, n$. The best bound so far is due to Lee and Vempala \cite{lee-vempala} after a breaktrough of Eldan \cite{eldan}: there is a universal constant $c$ such that for all log-concave probability measures on $\mathbb R^n$
  \[ C_P(\mu) \le c \|\mathrm{Cov}(\mu) \|_{HS}=c\big(\mathrm{Tr}(\mathrm{Cov}(\mu)^*\mathrm{Cov}(\mu))\big)^{1/2}.\]
  This implies that $\kappa_n\le c \sqrt n$.

\subsection{Log-concave measures with symmetries}

For a Borel measure $\mu$ on $\RR^n$ and a function $f \in L_2(\mu)$ we write
\begin{equation} \| f \|_{H^{-1}(\mu)} = \sup \left \{ \int_{\RR^n} f u d \mu \, ; \, u \in L^2(\mu) \textrm{ is locally-Lipschitz with }  \int_{\RR^n} |\nabla u|^2 d \mu \leq 1  \right \}. \label{eq_951} \end{equation}
 The norm $\| f \|_{H^{-1}(\mu)}$ makes sense only when $\int f d\mu = 0$, as otherwise
$\| f \|_{H^{-1}(\mu)} = +\infty$. By duality, it follows from the definition of the Poincar\'e constant that for any $f \in L^2(\mu)$ with $\int f d \mu = 0$,
\begin{equation} \| f \|_{H^{-1}(\mu)}^2 \leq C_P(\mu) \int_{\RR^n} f^2 d \mu.
\label{eq_915} \end{equation}
The following proposition is an extension of \cite[Lemma 1]{klartag-unconditional}, from uniform measures on $C^\infty$ smooth convex bodies to finite log-concave measures. A proof is provided for completeness.

\begin{proposition} Let $\mu$ be a finite, log-concave  measure on $\RR^n$. Let $f: \RR^n \rightarrow \RR$ be a locally-Lipschitz function in $L^2(\mu)$ with $\partial_i f \in L^2(\mu)$ and $\int \partial_i f d \mu = 0$ for all $i$. Then,
	\begin{equation} \mathrm{Var}_{\mu}(f) \leq \sum_{i=1}^n \| \partial_i f \|_{H^{-1}(\mu)}^2, \label{eq_941}
	\end{equation}
where we recall that $\mathrm{Var}_{\mu}(f) = \int (f - E)^2 d \mu$ and $E = \int f d \mu / \mu(\RR^n)$.	\label{prop_1004}
\end{proposition}

  We require the following lemma, whose proof appears in the Appendix below:

\begin{lemma} It suffices to prove Proposition \ref{prop_1004} under the  additional assumption that the measure $\mu$ has a $C^{\infty}$-smooth density in $\RR^n$ which is everywhere positive.
	\label{lem_1111}
\end{lemma}

\begin{proof}[Proof of Proposition \ref{prop_1004}] Thanks to Lemma \ref{lem_1111}, we may assume that
$\mu(dx) = \exp(-\psi(x)) dx$, where $\psi: \RR^n \rightarrow \RR$ is smooth and convex. We may also add a constant to $f$ and assume that $\int f d \mu = 0$. Define the associated Laplace operator
	$$ L u = \Delta u - \langle \nabla u, \nabla \psi \rangle = \sum_{i=1}^n \partial_{ii} u - \partial_i u \cdot \partial_i \psi $$
	for a $C^2$-smooth, compactly-supported $u: \RR^n \rightarrow \RR$. A virtue of this operator is the integration by parts
	$$ \int_{\RR^n} u (Lv) d \mu = -\int_{\RR^n} \langle \nabla u, \nabla v \rangle d \mu, $$
	valid whenever $v: \RR^n \rightarrow \RR$ is $C^2$-smooth and compactly-supported and $u$ is locally-Lipschitz. The Bochner formula states that for
	any $C^2$-smooth, compactly-supported
	function $u: \RR^n \rightarrow \RR$,
	$$ \int_{\RR^n} (Lu)^2 d \mu =
	\sum_{i=1}^n \int_{\RR^n} |\nabla \partial_i u|^2 d \mu + \int_{\RR^n} (\nabla^2 \psi) \nabla u \cdot \nabla u \, d \mu \geq \sum_{i=1}^n \int_{\RR^n} |\nabla \partial_i u|^2 d \mu. $$	
	This Bochner formula is discussed in \cite{CFM}, where it is also proven (see \cite[Lemma 3]{CFM}) that there exists a sequence of compactly-supported, $C^2$-smooth functions $u_k: \RR^n \rightarrow \RR \ (k=1,2,\ldots)$ with
\begin{equation} \lim_{k \rightarrow \infty} L u_k = f \qquad \text{in} \ L^2(\mu). \label{eq_401} \end{equation}
Now, for any $k \geq 1$,
\begin{align} \nonumber
\int_{\RR^n} f (L u_k) d \mu
& = -\sum_{i=1}^n \int_{\RR^n} \partial_i f \cdot \partial_i u_k d \mu \leq \sqrt{ \sum_{i=1}^n \int_{\RR^n} |\nabla \partial_i u_k|^2 d \mu} \cdot \sqrt{ \sum_{i=1}^n \| \partial_i f \|_{H^{-1}(\mu)}^2} \\ & \leq \| L u_k \|_{L^2(\mu)} \cdot \sqrt{ \sum_{i=1}^n \| \partial_i f \|_{H^{-1}(\mu)}^2}. \label{eq_400}
\end{align}
	By letting $k$ tend to infinity we deduce (\ref{eq_941}) from (\ref{eq_401}) and (\ref{eq_400}).
\end{proof}

Let us write $C_P(\mu, ``even")$ for the smallest number $C > 0$ for which
$$ \mathrm{Var}_{\mu}(f) \leq C \int_{\RR^n} |\nabla f|^2 d \mu $$
for all even, locally-Lipschitz functions $f \in L^2(\mu)$.
We write $C_P(\mu, ``odd")$ for the analogous quantity where $u$ is assumed an odd function.

\medskip When $\mu$ is an even measure and $f \in L^2(\mu)$ is odd,
we may restrict attention to odd functions $u$
in the definition (\ref{eq_951}) of $\| f \|_{H^{-1}(\mu)}$. Indeed, replacing $u(x)$ by its odd part $[u(x) - u(-x)] / 2$ cannot possibly increase $\int |\nabla u|^2 d \mu$ or affect the integral
$\int f u d \mu$ at all. Consequently, in this case,
\begin{equation}
\| f \|_{H^{-1}(\mu)}^2 \leq C_P(\mu, ``odd") \cdot \int_{\RR^n} f^2 d \mu. \label{eq_429}
\end{equation}
Moreover, when $\mu$ is an even measure in $\RR^n$ we have
\begin{equation} C_P(\mu) = \max \{ C_P(\mu, ``odd"), C_P(\mu, ``even") \}. \label{eq_1002} \end{equation}
This follows from the fact that any locally-Lipschitz $f \in L^2(\mu)$ may be decomposed as $f = g + h$ with $g$ even and $h$ odd, and $\int gh d \mu = \int (\nabla g \cdot \nabla h)  d \mu = 0$.
In the case where the even measure $\mu$ is additionally assumed log-concave, formula (\ref{eq_1002}) may be improved. The following corollary is an extension of \cite[Corollary 2(ii)]{klartag-unconditional} from  smooth convex bodies to finite log-concave measures. This extension  requires a modified argument, as the one in \cite{klartag-unconditional} was based on eigenfunctions, which may not exist in general.
\begin{corollary} \label{cor:even-sym} Let $\mu$ be a finite, log-concave  measure on $\RR^n$. Assume that $\mu$ is even. Then $$ C_P(\mu) = C_P(\mu, ``odd"). $$ \label{thm_1003}
\end{corollary}

\begin{proof} In view of (\ref{eq_1002}), we need to prove that $C_P(\mu, ``even") \leq C_P(\mu, ``odd")$. Thus, let $f \in L^2(\mu)$ be an even, locally-Lipschitz function. Then $\partial_i f$ is an odd function for all $i$. In the case where $\partial_i f \in L^2(\mu)$ for all $i$,
	by Proposition \ref{prop_1004} and by (\ref{eq_429}),
\begin{equation}
\mathrm{Var}_{\mu}(f) \leq \sum_{i=1}^n \| \partial_i f \|_{H^{-1}(\mu)}^2 \leq
C_P(\mu, ``odd") \cdot \sum_{i=1}^n
\int_{\RR^n} |\partial^i f|^2 d \mu =
C_P(\mu, ``odd") \cdot \int_{\RR^n} |\nabla f|^2. \label{eq_432}
\end{equation}
Note that (\ref{eq_432}) trivially holds when $\partial_i f \not \in L^2(\mu)$ for some $i$, as the right-hand side is infinite. Now (\ref{eq_432}) shows
that $C_P(\mu, ``even") \leq C_P(\mu, ``odd")$.
\end{proof}

\begin{proof}[Proof of Theorem \ref{th:interlace}] The first part of the theorem follows
from Corollary \ref{thm_1003}. As for the second part, let $E \subseteq L^2(\mu)$
be an $(n+1)$-dimensional subspace of locally-Lipschitz, odd functions. Consider the linear map $\theta: E \rightarrow \RR^n$ defined via
\begin{equation}
\theta(f) := \int_{\RR^n} \nabla f\,  d \mu.
\label{eq_412}
\end{equation}
Since $E$ is $(n+1)$-dimensional, there exists $0 \not \equiv f \in E$ with $\theta(f) = 0$. Since $f$ is odd,  the function $\partial_i f$ is an even function for all $i$. In the case where $\partial_i f \in L^2(\mu)$ for all $i$,
by Proposition \ref{prop_1004} and (\ref{eq_429}),
$$
\mathrm{Var}_{\mu}(f) \leq \sum_{i=1}^n \| \partial_i f \|_{H^{-1}(\mu)}^2 \leq
C_P(\mu, ``even") \cdot \sum_{i=1}^n
\int_{\RR^n} |\partial^i f|^2 d \mu =
C_P(\mu, ``even") \cdot \int_{\RR^n} |\nabla f|^2 d\mu. $$
This inequality trivially holds if $\partial_i f \not \in L^2(\mu)$ for some $i$. We have thus found $f \in E$ with
$$ \lambda_P(\mu, ``even") = \frac{1}{C_P(\mu, ``even")} \leq \frac{\int_{\RR^n} |\nabla f|^2 d \mu}{\mathrm{Var}_{\mu}(f)}, $$
completing the proof of the theorem.
\end{proof}

%

A measure $\mu$ in $\RR^n$ is unconditional if it is invariant under coordinate reflections, i.e., for any test function $\vphi$ and any  choice of signs,
$$ \int_{\RR^n} \vphi(\pm x_1,\ldots,\pm x_n) d \mu(x) = \int_{\RR^n} \vphi(x_1,\ldots,x_n) d \mu(x). $$
The following corollary is similar to
\cite[Corollay 2(i)]{klartag-unconditional} but it does not involve any regularity assumption:

\begin{corollary}\label{cor:uncond-sym} Let $\mu$ be a finite, log-concave measure on $\RR^n$. Assume that $\mu$ is unconditional. Then
	$$ C_P(\mu) = C_P(\mu, ``\textrm{odd in at least one coordinate}"), $$
	i.e., in the definition of $C_P(\mu)$ it suffices to consider functions $f(x_1,\ldots,x_n)$ for which there is an index $i$ such that $f$ is odd with respect to $x_i$.
\end{corollary}

\begin{proof} For $I \subseteq \Omega_n = \{ 1,\ldots, n \}$ we say that $f(x_1,\ldots,x_n)$ is of type $I$ if it is even with respect to $x_i$ for $i \in I$ and odd with respect to $x_i$ for $i \not \in I$. Any $f \in L^2(\mu)$ may be decomposed into a sum of $2^n$ functions, each of a certain type $I \subseteq \Omega_n$. Moreover, even without the log-concavity assumption we have
	\begin{equation} C_P(\mu) = \max_{I \subseteq \Omega_n} C_P(\mu, ``\textrm{functions of type I}").  \label{eq_452}
	\end{equation}
	All we need is to
	eliminate the case $I = \Omega_n$ from the maximum in (\ref{eq_452}). However, if $f$ is of type $\Omega_n$, then each function $\partial_i f$ is of  type $\Omega_n\setminus \{i\} $. We may thus rerun the argument in (\ref{eq_432}) and complete the proof.
\end{proof}

\subsection{The structure of the eigenspace}

We move on to discuss properties of eigenfunctions of log-concave measures with symmetries, following their investigation  in \cite{klartag-unconditional}.
We will consider a log-concave probability measure $d\mu(x)=e^{-\psi(x)} dx$ such that $\psi:\RR^n\to \RR$ is of class $C^2$ and $D^2\psi(x)>0$ for all $x$. The Poincar\'e inequality asserts that the non-zero eigenvalues of
$-L$, where $$ L=\Delta-\langle \nabla\psi, \nabla \rangle, $$ are at least $1/C_P(\mu)$. We assume here that $\lambda_\mu=\lambda_P(\mu) = 1/C_P(\mu)$ is actually an eigenvalue for $L$ and study the structure of the corresponding eigenspace $E_\mu:=\{ f\in L^2(\mu);  Lf=-\lambda_\mu f\}$.
Note that elliptic regularity ensures that eigenfunctions are $C^2$-smooth.
First, we put forward the key ingredient in \cite{klartag-unconditional}.
We reproduce the proof, for completeness.
\begin{lemma}\label{lem:injectivity}
Under the above assumptions, the linear map  $\theta:E_\mu\to \RR^n$
defined in (\ref{eq_412})
is injective. As a consequence $\dim E_\mu\le n$.
\end{lemma}
\begin{proof}
Assume $Lf=-\lambda_\mu f$ and $\int \nabla f d\mu=0$.  Then using integration by parts, the Poincar\'e inequality for the zero average functions $\partial_i f$ and the Bochner formula gives
\begin{align*}
\lambda_\mu \int f^2d\mu&= -\int f Lf \, d\mu = \int |\nabla f|^2 d\mu =\sum_i \mathrm{Var}_\mu(\partial_i f) \le \frac{1}{\lambda_\mu}
\sum_i \int |\nabla \partial_i f|^2 d\mu\\
=& \frac{1}{\lambda_\mu}\left( \int (Lf)^2d\mu - \int \langle D^2\psi \nabla f,\nabla f\rangle d\mu \right)\le \frac{1}{\lambda_\mu} \int (Lf)^2d\mu =\lambda_\mu \int f^2 d\mu.
\end{align*}
Hence all the above inequalities are actually equalities. In particular
$ \int \langle D^2\psi \nabla f,\nabla f\rangle d\mu =0$, from which we conclude that $f$ is constant. Hence $0=Lf=-\lambda_\mu f$, and $f=0$.
\end{proof}
Let $\mathcal O_n$ be the group of linear isometries of the Euclidean space $\RR^n$. We consider the subgroup of isometries  which leave
$\mu$ invariant:
\[\mathcal O_n(\mu):=\big\{R\in \mathcal O_n;\; R\mu=\mu\big\}= \big\{R\in \mathcal O_n;\;  \psi\circ R=\psi\big\}.\]
\begin{lemma}\label{lem:Omu}
If $f\in E_\mu$ and $R\in \mathcal{O}_n(\mu)$ then $f\circ R^{-1}\in E_\mu$ and
\[ \theta\big( f\circ R^{-1}\big)= R \theta(f).\]
\end{lemma}
\begin{proof}
The fact that $f\circ R^{-1}$ is still an eigenfunction is readily checked. Next
\[ \theta\big( f\circ R^{-1}\big)=\int \nabla (f\circ R^{-1})d\mu
= \int R (\nabla f)\circ R^{-1} d\mu= R \int \nabla f \, d\mu,\]
where we have used that $R^{-1}$ is also the adjoint of $R$, and the invariance of $\mu$.
\end{proof}
\begin{rem}
This result can be formulated in a more abstract way. The group $\mathcal{O}_n(\mu)$ has a natural representation as  operators on  $\RR^n$, denoted by $\rho$. It has another one as operators on $E_\mu$,
denoted by $\pi$ and defined for $R\in \mathcal O_n(\mu)$ and $f\in E_\mu$ by $\pi(R)f=f\circ R^{-1}$.  The statement of the lemma means that
$\theta:E_\mu\to \RR^n$ intertwines $\pi$ and $\rho$.
\end{rem}

\begin{rem}
The arguments of the above two proofs  were used in \cite{klartag-unconditional} to establish the existence of antisymmetric eigenfunctions, more specifically of an odd eigenfunction when $\psi$ is even, and of an eigenfunction which is odd in one coordinate when $\psi$ is unconditional. Note that these results give Corollary~\ref{cor:even-sym} and also Corollary \ref{cor:uncond-sym} below under strong assumptions on the existence of eigenfunctions, which we could remove in the present paper. It was proven in \cite{barthe-cordero} that the existence of antisymmetric eigenfunctions extends as follows: if
there exist $R_1,\ldots, R_k\in \mathcal O_n(\mu)$ such that $\{x\in \RR^n;\; \forall i, R_ix=x\}=\{0\}$ then for every $f\in E_\mu\setminus \{0\}$ there exists $i$ such that $f\circ R_i -f\in E_\mu\setminus \{0\}$.
The proof of this is easy  from the lemmas:  it is always true that $f\circ R_i -f\in E_\mu$. Assume by contradiction that for all $i$, $f\circ R_i -f=0$. Then $\theta(f)=\theta(f\circ R_i)=R_i^{-1}\theta(f)$. So $\theta(f)\in \RR^n$ is a fixed point of all the $R_i$'s. By hypothesis, $\theta(f)=0$ hence $f=0$. 
\end{rem}
The above two statements allow to derive some more structural properties of $E_\mu$ when the measure has enough symmetries.

\begin{theo}
With the above notation, assume that $\mathcal O_n(\mu)$ has no non-trivial invariant subspace. Then the map $\theta$ is bijective. In particular
$\dim E_\mu=n$.
Moreover, for any $f\in E_\mu\setminus\{0\}$,
\[E_\mu=\mathrm{span}\big\{ f\circ R; \, R\in \mathcal O_n(\mu)\big\}.\]
\label{theo_427}
\end{theo}

\begin{proof}
By the above lemma, the range of $\theta$ is invariant by $\mathcal O_n(\mu)$. By Lemma~\ref{lem:injectivity}, the map $\theta$ is injective, so its  range cannot be reduced
to $\{0\}$. Therefore $\theta(E_\mu)=\RR^n$, i.e $\theta$ is surjective, hence bijective.

\medskip
Next consider $S:=\mathrm{span}\big\{f\circ R^{-1};\, R\in \mathcal O_n(\mu)\big\}\subset E_\mu$. Then, thanks to the latter lemma,
$\theta(S)=\mathrm{span}\big\{ R\theta(f);\, R\in \mathcal O_n(\mu)\big\}$ is $\mathcal O_n(\mu)$ invariant and non-zero. Therefore it is equal to $\RR^n$. Hence $S=E_\mu$.
\end{proof}

Theorem \ref{theo_427_} above  follows from Theorem \ref{theo_427}, as it is well-known by spectral theory that a locally-Lipschitz
function $f \in L^2(\mu)$ with $\int f d \mu = 0$ for which an equality in the Poincar\'e inequality is attained, belongs to $E_{\mu}$.

\medskip Eventually, let us give an example in a specific case: assume that $\mu$
has the symmetries of the cube, or equivalently that $\psi(x)=\psi\big(|x_{\sigma(1)}|,\ldots, |x_{\sigma(n)}|\big)$ for all permutations $\sigma$ of $\{1,\ldots,n\}$ and all $x\in \RR^n$. Then $\mathcal O_n(\mu)$ has no non-trivial invariant subspace and the above proposition applies. But one  can give a more precise description of the $n$-dimensional space $E_\mu$ in this case.

\medskip
Denote by $(e_i)_{i=1}^n$ the canonical basis of $\RR^n$, by $S_i$ the orthogonal symmetry with respect to the hyperplane $\{x; x_i=0\}$, and
$T_{ij}$, $i\neq j$ the linear operator on $\RR^n$ the action of which on the canonical basis is to exchange $e_i$ and $e_j$. Note that $S_i$ and $T_{ij}$ belong to $\mathcal O_n(\mu)$ and are involutive.
Since $\theta$ is bijective we define $f_i:=\theta^{-1}(e_i)$, and obtain a
basis $(f_i)_{i=1}^n$ of $E_\mu$. The relationships between vectors of $\RR^n$ and isometries in $\mathcal O_n(\mu)$ can be transfered to eigenfunctions thanks to $\theta$:
\begin{align*}
\theta(f_1)&=e_1=-S_1e_1=-S_1\theta(f_1)=\theta(-f_1\circ S_1)\\
\theta(f_1)&=e_1=S_ie_1=S_i\theta(f_1)=\theta(f_1\circ S_i),\quad \mathrm{if} \, i\neq 1\\
\theta(f_1)&=e_1=T_{ij}e_1=T_{ij}\theta(f_1)=\theta(f_1\circ T_{ij}),\quad \mathrm{if} \, i,j\neq 1
\end{align*}
imply that $f_1=-f_1\circ S_1$ and for $i,j\neq 1$, $f_1=f_1\circ S_i=f_1\circ T_{i,j}$.
In other words for any $(x_2,\ldots,x_n)$, the map $x_1\mapsto f_1(x_1,\ldots,x_n)$ is odd and for any $x_1$, the map
$(x_2,\ldots, x_n)\mapsto f_1(x_1,\ldots,x_n)$ is invariant by changes of signs and permutations of coordinates.  Still for $i\neq 1$,
\[ \theta(f_i)=e_i=T_{1i}e_1=T_{1i} \theta(f_1)= \theta(f_1\circ T_{1i}) \]
yields $f_i=f_1\circ T_{1i}$. In particular, $f_i$ is an odd function of $x_i$
and an unconditional and permutation invariant function of $(x_j)_{j\neq i}$. Consequently for $i\neq j$, $\int f_i f_j \, d\mu=0$ (the integral against $dx_i$ is equal to zero since $f_i$ is odd in $x_i$ while $f_j$ and $\psi$ are even in $x_i$).
Summarizing, $(f_1,f_1\circ T_{12},\ldots, f_1\circ T_{1n})$ is an orthogonal basis of $E_\mu$.

\section{Perturbed products}
In this section we investigate Poincar\'e inequalities for multiplicative perturbations of product measures.

\subsection{Unconditional measures}

We now describe a comparison result  which may be viewed as the higher-dimensional analog of
Proposition \ref{prop:unimodal-perturbation}, in the case of product measures.
We write $\RR^n_+ = [0, \infty)^n$.
\begin{theo}\label{theo:before-rem-cube}
	For $i=1,\ldots,n$, let  $d\mu_i(t)=\mathbf{1}_{(-b_i,b_i)}(t) e^{-V_i(t)}\, dt$ be an origin symmetric probability measure on $\mathbb  R$, with $b_i\in(0,\infty]$ and $V_i$ continuous on $\mathbb R$. Let $\rho:\mathbb R^n\to \mathbb R^+$  be such that $d\mu^{n,\rho}(x)=\rho(x)\prod_{i=1}^nd\mu_i(x_i)$ is a probability measure. Assume that $\rho$ is unconditional (i.e. $\rho(x_1,\ldots,x_n)=\rho(|x_1|,\ldots, |x_n|)$ for all $x\in \mathbb R^n$) and coordinatewise non-increasing on $\mathbb R_+^n$.
	If in addition $\mu^{n,\rho}$ is log-concave, then
	\[C_P(\mu^{n,\rho})\le C_P(\mu^{n,1})=\max_i C_P(\mu_i).\]
	This holds in particular when the measures $\mu_i$ are even and log-concave and $\rho$ is log-concave and unconditionnal.
\end{theo}
\begin{proof}
	Since $\mu^{n,\rho}$ is log-concave and unconditional, we know by Corollary~\ref{cor:uncond-sym} that it is enough to prove the Poincar\'e inequality for functions
	which are odd with respect to one coordinate. Let $f:\mathbb R^n\to \mathbb R$ be locally Lipschitz, and assume that it is odd in the first variable (other variables are dealt with in the same way). Then by symmetry $\int f \, d\mu^{n,\rho}=0$, so that
	\[ \mathrm{Var}_{\mu^{n,\rho}}(f)= \int f^2 d\mu^{n,\rho}=\int \left( \int_{\mathbb R} f^2(x) \rho(x) d\mu_1(x_1)\right) \prod_{i\ge 2} d\mu_i(x_i). \]
	In the sense of   $\mu_2\otimes\cdots\otimes \mu_n$, for almost every $\overline{x}:=(x_2,\ldots,x_n)$, $Z_{\overline{x}}:=\int_{\mathbb R}
	\rho(x) d\mu_1(x_1)<+\infty$. Thus way may consider the probability measure $\rho(x_1,\overline{x}) d\mu_1(x_1)/ Z_{\overline{x}}$. It is
	a perturbation of an even probability measure on $\mathbb R$, by the even unimodal function $x_1\mapsto \rho(x_1,\overline{x})/Z_{\overline{x}}$. Hence by
	Proposition \ref{prop:unimodal-perturbation}, its Poincar\'e constant is at most $C_P(\mu_1)$. Since $x_1 \mapsto f(x_1,\overline{x})$ is odd,
	it has a zero average for the later measure and we get
	\[ \int_{\mathbb R} f^2(x) \rho(x)\frac{d\mu_1(x_1)}{Z_{\overline{x}}} \le C_P(\mu_1) \int_{\mathbb R} (\partial_1 f(x))^2  \rho(x)\frac{d\mu_1(x_1)}{Z_{\overline{x}}}.\]
	Cancelling $ Z_{\overline{x}}$ and plugging in the former equality, we get
	\[ \mathrm{Var}_{\mu^{n,\rho}}(f)\le \int \left( C_P(\mu_1) \int_{\mathbb R} (\partial_1 f(x))^2 \rho(x) d\mu_1(x_1)\right) \prod_{i\ge 2} d\mu_i(x_i)\le  \max_i C_P(\mu_i) \int |\nabla f|^2  d \mu^{n,\rho}.
	\qedhere
	\]
\end{proof}
\begin{rem}
	The hypothesis of unconditionality on the perturbation $\rho$ cannot be dropped, as the following example shows. Denote by $U([a,b])$ the uniform probability measure on $[a,b]$. Classically, $C_P(U([a,b]))=(b-a)^2/\pi^2$. We choose $\mu_i=U([-\frac12,\frac12])$. Then the measure $\mu^{n,1}$ is uniform on the unit cube $C_n:=[-\frac12,\frac12]\subset \mathbb R^n$, and $C_P(\mu^{n,1})=\pi^{-2}$.
	Let $\varepsilon \in (0,1)$ and consider an orthogonal parallelotope $P_\varepsilon$ included in the cube $C_n$ and of maximal side length $(1-\varepsilon )\sqrt{n}$ (such parallelotopes are easily constructed. When $\varepsilon $ tends to zero they collapse to the main diagonal of the cube, the length of which is $\sqrt n$). Then define $\rho_\varepsilon=\mathbf{1}_{P_\varepsilon}/\mathrm{Vol}(P_\varepsilon)$. Clearly $\mu^{n,\rho_\varepsilon}$ is the uniform measure on $P_\varepsilon$, which is a product measure. So by the tensorisation property
	$C_P(\mu^{n,\rho_\varepsilon})=\frac{1}{\pi^2} ((1-\varepsilon)\sqrt n)^2$.
\end{rem}

\begin{rem} The product hypothesis is also important. Consider  the uniform measure $U(\sqrt{n}B_2^n)$ on the Euclidean Ball of radius $\sqrt n$ in $\mathbb R^n$, for $n\ge 2$. It is well-known  that $\sup_n C_P(U(\sqrt{n}B_2^n))<+\infty$. For $\varepsilon\in (0,1)$, define the unconditional parallelotope
	\[Q_\varepsilon=\left\{x\in\mathbb R^n; |x_1|\le \sqrt{n-\varepsilon}\; \mathrm{and}\, \forall i\ge 2, |x_i|\le \sqrt{\frac{\varepsilon}{n-1}} \right\}\subset \sqrt{n} B_2^n.\]
	Since it is a product set, $C_P(U(Q_\varepsilon))=C_P\big(U([-\sqrt{n-\varepsilon}, \sqrt{n-\varepsilon}])\big)=(n-\varepsilon)/\pi^2$. Hence  $U(Q_\varepsilon)$ is an unconditional and log-concave perturbation of $U(\sqrt{n}B_2^n)$, which is itself log-concave and unconditional. Nevertheless  the former  has a much larger Poincar\'e constant than the latter when the dimension grows. See also Section 3.3 below.
\end{rem}

\subsection{The general case}
The above examples show that a dimension dependence is sometimes needed, of order $n$ for the covariances and Poincar\'e constants. We show next that this is as bad as it gets, and that such a control of the covariance can be obtained independently of the even log-concave perturbation.

\begin{theo}
Let $\mu_1, \ldots,\mu_n$ be even log-concave probability measures on $\mathbb R$, and let $\rho:\R^n\to \R^+$ be an even log-concave function such that
\[  d\mu_\rho(x):= \rho(x) \prod_{i=1}^n d\mu_i(x_i),\quad x\in \R^n\]
is a probability measure.  Then, covariance matrices can be compared:
\[ \mathrm{Cov}(\mu^{n,\rho})\le n \, \mathrm{Cov}(\mu^{n,1}).\]
Moreover,
\[ C_p(\mu^{n,\rho}) \le c \sum_{i=1}^n \mathrm{Var}(\mu_i)
\le  c\, n \max_{i} C_P(\mu_i)=c\, n \, C_P(\mu^{n,1}),\]
where $c$ is a universal constant.
\end{theo}
\begin{proof}
We start with the covariance inequality.
Set $\sigma_i^2=\mathrm{Var}(\mu_i)$. Let $g$ be an even log-concave function on $\mathbb R$. Then since $g$ is non-increasing on $\mathbb R^+$,
\[ \int_{\mathbb R} t^2 g(t) \, d\mu_i(t) \le \left( \int_{\mathbb R} t^2   d\mu_i(t)\right)  \left( \int_{\mathbb R}  g(t) \, d\mu_i(t)\right) .\]
Indeed, by symmetry this follows from  the basic fact that $2\mathrm{cov}_{m}(f,g)=\int_{(R_+)^2} (f(x)-f(y))(g(x)-g(y)) \, dm(x) dm(y)\le 0$ if $m$ is a probability measure on $\mathbb R^+$, $f$ is non-decreasing and $g$ is non-increasing.
The above inequality, sometimes referred to as Chebyshev's sum inequality, can be restated in terms of the peaked ordering as
$ t^2 d\mu_i(t)\prec \sigma_i^2 \mu_i $.
Such an inequality is preserved by taking on both side the tensor product with an even log-concave measure (e.g. Kanter \cite[Corollary 3.2]{Kanter}).
Hence, tensorizing with $\otimes_{j\neq i} \mu_j$
\[     x_i^2  d\mu_1(x_1)\ldots d\mu_n(x_n) \prec \sigma_i^2 \mu_1\otimes\cdots\otimes \mu_n  .\]
This means that the left-hand side measure has smaller integral against even log-concave functions. Applying this with $\rho$ gives
\begin{equation}
 \label{eq:bound-coordinate-variance}
\int x_i^2 d\mu^{n,\rho}(x)\le \sigma_i^2.
\end{equation}
This is enough to upper bound the covariance matrix. Indeed, for $\theta\in \mathbb R^n$,
\begin{align*}
\mathrm{Var}_{\mu^{n,\rho}}(\langle \cdot, \theta\rangle)& =\int \langle x,\theta\rangle^2 d\mu^{n,\rho}(x)=\sum_{i,j} \int x_ix_j \theta_i\theta_j \, d\mu^{n,\rho}(x) \\
& \le \sum_{i,j} |\theta_i|\, |\theta_j| \left(\int x_i^2 d\mu^{n,\rho}(x) \right)^{\frac12}\left(\int x_j^2 d\mu^{n,\rho}(x) \right)^{\frac12} \\
& \le  \sum_{i,j}\sigma_i  |\theta_i| \sigma_j |\theta_j|   = \left( \sum_{i=1}^n  |\theta_i|  \sigma_i\right) ^2 \\
& \le n \sum_{i=1}^n  \sigma_i^2 \theta_i^2= n \, \mathrm{Var}_{\mu_1\otimes\cdots\otimes \mu_n}(\langle \cdot, \theta\rangle) .
\qedhere
\end{align*}
Eventually, since $\mu^{n,\rho}$ is log concave, we may apply Inequality \eqref{eq:trace-bound}
\[C_P(\mu^{n,\rho})\le c \, \mathrm{Tr}\big(\mathrm{Cov}(\mu^{n,\rho})\big)= c\sum_{i=1}^n \int x_i^2 d\mu^{n,\rho}(x).\]
We conclude thanks to \eqref{eq:bound-coordinate-variance}.
\end{proof}

\subsection{Gaussian mixtures}
\label{sec_GM}

In this section, we consider $n$ probability measures on $\R$ which are absolutely continuous {\it Gaussian mixtures}. This means that  $\mu_i(dt)=\varphi_i(t)\, dt$ for $i=1,\ldots,n$ with
\begin{equation}
\varphi_i(t)=  \int_{\R_+^*} \frac{e^{-\frac{t^2}{2\sigma^2}}}{\sigma \sqrt{2\pi}} \, dm_i(\sigma) ,\quad t\in \R \label{eq_546}
\end{equation}
where $m_i$ is a probability measure
and $\RR^*_+ = (0, \infty)$. In other words if $R_i$ is a random variable with law $m_i$ and is independent of a standard Gaussian variable $Z$, then the product $R_iZ$ is distributed according to $\mu_i$. These measures were considered by Eskenazis, Nayar and Tkocz \cite{ENT-mixtures}, who showed that several geometric and entropic properties of Gaussian measures  extend to Gaussian mixtures.

\subsubsection{Using the covariance}

For log-concave probability measures, it is known that the Poincar\'e constant is related to the operator norm of the covariance matrix of the measure.
In order to estimate the covariance, we use an extension by Eskenazis, Nayar and Tkocz of the Gaussian correlation inequality due to Royen \cite{royen}.
A function $f$ is quasi-concave if its upper level sets $\{ x ; f(x) \geq t \}$
are convex for all $t$.

\begin{theo}[\cite{ENT-mixtures}]
Let $\mu_1,\ldots,\mu_n$ be probability measures on $\mathbb R$ which all are Gaussian mixtures.  Let $f,g:\mathbb R^n\to \mathbb R^+$
be even and quasi-concave, then for $\mu=\mu_1\otimes\cdots\otimes \mu_n$,
\[ \int fg \, d\mu \ge \left( \int f\, d\mu\right) \left( \int g\, d\mu\right).\]
\label{thm_548}
\end{theo}

\begin{rem}
The inequality is actually valid for the more general class of even and unimodal functions (i.e. increasing limits of positive combinations of indicators of origin symmetric convex sets).
\end{rem}

For our purpose we rather need a weaker version of Theorem \ref{thm_548}. Let $c:\mathbb R^n\to \R$ be an even convex function, and $g$ be even and log-concave; for $\varepsilon>0$, consider the log-concave function $f=\exp(-\varepsilon g)$.
Then the above theorem gives
$\int e^{-\varepsilon c}g \, d\mu \ge \left(\int e^{-\varepsilon c} \, d\mu\right)
\left( \int g \, d\mu\right)$. There is equality for $\varepsilon=0$, so comparing derivatives at $\varepsilon =0$ yields that an even convex and an even log-concave function are negatively correlated for $\mu$:
\begin{equation}\label{eq:convex-logconcave}
 \int cg \, d\mu \le \left( \int c\, d\mu\right) \left( \int g\, d\mu\right).
\end{equation}
In the case of centered Gaussian measures, this negative correlation property between even convex and even log-concave functions was established first by Harg\'e \cite{harge}.

\begin{proposition}\label{th:GM-odd-functions}
Let $\mu_1,\ldots,\mu_n$ be Gaussian mixtures, and let $\rho:\mathbb R^n\to \mathbb R^+$ be an even log-concave function such that the measure $d\mu^{n,\rho}(x)=\rho(x) \prod_{i=1}^n d\mu_i(x_i)$ is a probability measure on $\mathbb R^n$.  Then
\[ \mathrm{Cov}(\mu^{n,\rho})\le \mathrm{Cov}(\mu^{n,1})\] 
If in addition $\mu^{n,\rho}$ is log-concave (which is true if the measures  $\mu_i$ and the function $\rho$ are log-concave), then
\[ C_p(\mu^{n,\rho}) \le c\, n^{\frac12} \max_{i} \mathrm{Var}(\mu_i)
\le  c\, n^{\frac12} \max_{i} C_P(\mu_i)=c\, n^{\frac12} C_P(\mu^{n,1}),\]
where $c$ is a universal constant.
\end{proposition}
\begin{proof}
Let $\theta\in \mathbb R^n$. Since $x\mapsto \langle x,\theta\rangle^2$ is even and convex, the correlation inequality \eqref{eq:convex-logconcave} yields
\[ \int  \langle x,\theta\rangle^2 \rho(x) \prod d\mu_i(x_i) \le \left(\int \langle x,\theta\rangle^2  \prod d\mu_i(x_i) \right)
\int  \rho(x) \prod d\mu_i(x_i).
\]
Since the measures are centered, this can  be rewritten as
\[\mathrm{Var}_{\mu^{n,\rho}} \big( \langle \cdot, \theta\rangle\big)
\le \mathrm{Var}_{\mu^{n,1}} \big( \langle \cdot, \theta\rangle\big),\quad \theta \in \mathbb R^n.\]
Hence the covariance inequality is proved. For the second part of the statement, we apply the best  general result towards the Kannan-Lovasz-Simonovits conjecture, recalled in Section
\ref{sec:KLS}: for  every log-concave probability measure $\eta$ on $\mathbb R^n$,
$C_P(\eta)\le c\, n^{1/2} \|\mathrm{Cov}(\eta)\|_{op}$.
\end{proof}
\begin{rem}
The KLS conjecture predicts that for some universal constant $\kappa$ and for all log-concave probability measures $\eta$, $C_P(\eta)\le \kappa  \|\mathrm{Cov}(\eta)\|_{op}$.
If it were confirmed, then the conclusion of the above theorem could be improved to $C_P(\mu^{n,\rho})\le \kappa \, C_P(\mu^{n,1})$.
\end{rem}

\begin{rem}
The correlation inequality proves that $\mu^{n,\rho}\succ \mu^{n,1}$
for the peaked ordering on measures: $\mu\succ \nu$ means $\mu(K)\ge \nu(K)$ for all origin-symmetric convex sets, and imples $\int f d\mu\ge \int f d\nu$ for all (even) unimodal functions. Also, the weaker correlation inequality \eqref{eq:convex-logconcave} implies that $\mu^{n,\rho}$ is dominated by $\mu^{n,1}$ in the Choquet ordering (integrating against convex functions).
\end{rem}

\subsubsection{Direct approach}
Working directly on the Poincar\'e inequality, we will improve the $n^{1/2}$ to $\log(n)$ in Proposition \ref{th:GM-odd-functions}.

\begin{lemma}\label{lem:poincCM}
Let $\mu_1, \ldots,\mu_n$ be Gaussian mixtures as in (\ref{eq_546}), and let $\rho:\R^n\to \R^+$ be an even log-concave function such that
\[  d\mu^{n,\rho}(x):= \rho(x) \prod_{i=1}^n d\mu_i(x_i),\quad x\in \R^n\]
is a probability measure. Then for every odd and locally Lipschitz function $f:\R^n\to \R$, it holds
\[ \mathrm{Var}_{\mu^{n,\rho}}(f)\le \int \sum_{i=1}^n \alpha_i(x_i) (\partial_i f(x))^2 d\mu^{n,\rho}(x),\]
where
\[ \alpha_i(t):= \frac{1}{\varphi_i(t)} \int_{\R_+^*} \sigma\frac{e^{-\frac{t^2}{2\sigma^2}}}{ \sqrt{2\pi}} \, dm_i(\sigma) =\frac{1}{\varphi_i(t)}\int_{|t|}^{+\infty} u \varphi_i(u)\, du,\quad t\in \R.\]
\end{lemma}
\begin{proof}
Since $f$ is odd and $\mu^{n,\rho}$ has an even density,
\begin{align*}
 \mathrm{Var}_{\mu^{n,\rho}}(f)& =\int f^2 \, d\mu^{n,\rho}= \int f^2(x)\rho(x) \prod_{j=1}^n \left( \int_{\R_+^*} \frac{e^{-\frac{x_j^2}{2\sigma_j^2}}}{\sigma_j \sqrt{2\pi}} \, dm_j(\sigma_j) \right) \,  dx \\
 &= \int_{(\R_+^*)^n} \left( \int_{\mathbb R^n} f^2(x) \rho(x) e^{-\frac12 \sum_j \frac{x_j^2}{\sigma_j^2}} \frac{dx}{(2\pi)^{n/2}\prod_j \sigma_j}\right) \prod_{j=1}^n dm_j(\sigma_j)
\end{align*}
For each $(\sigma_i)_i$ we estimate the inner integral from above thanks to the Brascamp-Lieb inequality, applied to the probability measure
 \[ dM_\sigma(x)= \frac{1}{Z_\sigma} \rho(x) e^{-\frac12 \sum_j \frac{x_j^2}{\sigma_j^2}} \frac{dx}{(2\pi)^{n/2}\prod_i \sigma_j}\]
Since $M_\sigma$ is  log-concave with respect to the Gaussian measure
$x \mapsto \exp(-\frac 12 \langle
\mathrm{Diag}(\sigma)^2 x, x \rangle)$,  the Brascamp-Lieb inequality
in the form of Corollary \ref{cor:BL} gives
\[ \mathrm{Var}_{M_\sigma}(f)\le \int \left\langle \mathrm{Diag}(\sigma)^2 \nabla f, \nabla f\right\rangle \, dM_\sigma(x).\]
Since $f$ is odd and $M_\sigma$ is an even measure, we obtain that
$\int f^2 dM_\sigma \le \int (\sum \sigma_i^2 (\partial_i f)^2) \, dM_\sigma$. Observe that in this formulation, the normalizing constant $Z_\sigma$ appears on both sides and therefore cancels.  This leads to
\begin{align*}
 \mathrm{Var}_{\mu_\rho}(f)&
\le  \int_{(\R_+^*)^n} \left( \int_{\mathbb R^n} \Big( \sum_i \sigma_i^2 \big(\partial_i f(x)\big)^2\Big) \rho(x) e^{-\frac12 \sum_j \frac{x_j^2}{\sigma_j^2}} \frac{dx}{(2\pi)^{n/2}\prod_j \sigma_j}\right) \prod_{j=1}^n dm_j(\sigma_j)\\
&=\sum_{i=1}^n \int_{\mathbb R^n}   \big(\partial_i f(x)\big)^2 \left(\int_{(\R_+^*)^n}  \sigma_i^2 \prod_{j=1}^n \left(e^{-\frac{x_j^2}{2\sigma_j^2}} \,\frac{ dm_j(\sigma_j)}{\sigma_j\sqrt{2\pi}} \right) \right) \rho(x)\, dx \\
&=\sum_{i=1}^n \int_{\mathbb R^n}   \big(\partial_i f(x)\big)^2 \left(\int_{(\R_+^*)^n}  \sigma_i  e^{-\frac{x_i^2}{2\sigma_i^2}} \,\frac{ dm_i(\sigma_i)}{\sqrt{2\pi}}  \right) \prod_{j\neq i}\varphi_j(x_j)\, \rho(x)\, dx \\
&=\sum_{i=1}^n \int_{\mathbb R^n}   \big(\partial_i f(x)\big)^2 \alpha_i(x_i) \,  d\mu^{n,\rho}(x).
\end{align*}
It remains to check the validity of the second expression of $\alpha_i$. This is obvious from the definition of $\varphi_i$ after interchanging integrals as follows:
\[ \int_{|t|}^{+\infty} u\varphi_i(u) du = \int_0^{+\infty} \left(\int_{|t|}^{+\infty}  u e^{-\frac{u^2}{2\sigma^2}} du\right) \frac{dm_i(\sigma)}{\sigma\sqrt{2\pi}} = \int_0^{+\infty}   \sigma^2  e^{-\frac{t^2}{2\sigma^2}}  \frac{1}{\sigma\sqrt{2\pi}} \, dm_i(\sigma) .\]
\end{proof}

\begin{lemma}\label{lem:bound-coef}
Let $\varphi:\mathbb R\to \mathbb R^+$ be an even and log-concave function such that $\int \varphi=1$. Then for all $t\in \mathbb R$,
\[  \frac{\int_{|t|}^{+\infty} u\varphi(u) du}{\varphi(t)}\le \frac{|t|}{2\varphi(0)}+\frac{1}{4\varphi(0)^2}\cdot\]
There is equality when for some $\lambda >0$ and for all $u$, $\varphi(u)=\lambda \exp(-\lambda|u|)/2$.
\end{lemma}
\begin{proof}
It is enough to deal with all $t\ge 0$. For such a fixed $t$, set  for all $v>0$, $\psi(v):=\varphi(t+v)$. Then changing variables by $ u=t+v$
\[\frac{\int_{t}^{+\infty} u\varphi(u) du}{\varphi(t)} = \frac{\int_0^{+\infty} (t+v)\psi(v) dv}{\psi(0)} = t\frac{\int_0^{+\infty}\psi}{\psi(0)}+\frac{\int_0^{+\infty} v\psi(v)dv}{\psi(0)}.\]
Since $\psi$ is log-concave, the Berwald-Borell inequality implies that the function
\[p>0 \mapsto G(p):=\left(\frac{1}{\psi(0)\Gamma(p)} \int_0^{+\infty} \psi(u) u^{p-1}du \right)^{\frac1p}\]
is non-increasing (see \cite{milman-pajor} or e.g. Theorem 2.2.3 in \cite{BOOKgreek}). The inequality $G(1)\ge G(2)$ allows us to deduce that
\[\frac{\int_{t}^{+\infty} u\varphi(u) du}{\varphi(t)} \le  t\frac{\int_0^{+\infty}\psi}{\psi(0)}+\left( \frac{\int_0^{+\infty}\psi}{\psi(0)}\right)^2.\]
With our notation
\[ \frac{\psi(0)}{\int_0^{+\infty}\psi}=\frac{\varphi(t)}{\int_t^{+\infty} \varphi}=-\frac{d}{dt}\log \left(\int_{t}^{+\infty} \varphi\right).\]
Since $\varphi$ is log-concave, the Pr\'ekopa-Leindler inequality ensures
that the tail function $t\mapsto \log \Big(\int_t^{+\infty}\varphi\Big)$ is concave, and thus has a non-increasing derivative. It follows that for $t>0$,
\[\frac{\int_0^{+\infty}\psi}{\psi(0)} =\frac{\int_t^{+\infty} \varphi}{\varphi(t)} \le \frac{\int_0^{+\infty} \varphi}{\varphi(0)}=\frac{1}{2\varphi(0)} \cdot \]
This leads to the claimed inequality. The case of equality is checked by direct calculations.
\end{proof}
\begin{rem}
Better estimates depending on $\varphi$ are easily established. If the even probability density is given by $\varphi=e^{-V}$ where $V$ is differentiable, even and  convex, then for $t>0$,
\[  \frac{d}{dt}\left( -\frac{te^{-V(t)}}{V'(t)}\right) = te^{-V(t)}\left( 1+\frac{V"(t)}{V'(t)}-\frac{1}{tV'(t)}\right) \ge  te^{-V(t)}\left( 1-\frac{1}{tV'(t)}\right). \]
Integrating, we obtain that for $t > 0$ such that $t V'(t) \geq 2$, 
\begin{equation}  \int_t^{+\infty} u \varphi(u)\, du \le 2  \frac{t}{V'(t)} \varphi(t), 
\label{eq_star} \end{equation}
since in this case also $s V'(s) \geq 2$ for all $s \geq t$.
\end{rem}

\begin{theo}\label{th:GM-poincare}
For $i=1,\ldots,n$, let $\mu_i(dt) =\varphi_i(t)\, dt$  be a Gaussian mixture on $\mathbb R$ which is  log-concave. Let $\rho:\mathbb R^n\to \mathbb R^+$ be an even log-concave function such that $d\mu^{n,\rho}(x)=\rho(x) \prod_{i=1}^n d\mu_i(x_i)$ is a probability measure on $\mathbb R^n$.
Then
\[ C_P(\mu^{n,\rho}) \le (1+C \log  n) \, C_P(\mu^{n,1}) =   (1+C \log  n) \max_{i} C_P(\mu_i) ,\]
where $C$ is a universal constant.
\end{theo}
\begin{proof}
The case $n=1$ is a direct application of Proposition \ref{prop:unimodal-perturbation}. Next we focus on  $n\ge2$.
We follow the truncation strategy from \cite{klartag-unconditional}.
Let $X_i$ be a random variable of law $\mu_i$. Since the latter is symmetric and log-concave,
classical results due to Borell and Hensley (see \cite{milman-pajor} or
Chapter 2 in \cite{BOOKgreek}) give
\[ \| X_1\|_{\psi_1}\le c \| X_i\|_2 \le \frac{c}{\sqrt 2 \varphi_i(0)},\]
where the Orlicz norm involves $\psi_1(t)=e^{|t|}-1$ and $c > 0$ is explicit and universal. Choose $\varepsilon:= \sqrt{2}/c$. The later inequality implies $\mathbb E \exp\big(\varepsilon \varphi_i(0) |X_i|\big)\le 2$.

\medskip
By the correlation inequality \eqref{eq:convex-logconcave}, and then Jensen's inequality
\begin{align*}
&\exp\left( \varepsilon \int \max_i \big(|x_i|\varphi_i(0)\big) \, d\mu^{n,\rho}(x) \right)\le
 \exp\left( \varepsilon \int \max_i \big(|x_i|\varphi_i(0)\big) \, d\mu^{n,1}(x) \right)\\
\le &  \int \exp\left( \varepsilon \max_i \big(|x_i|\varphi_i(0)\big)\right) \, d\mu^{n,1}(x)
\le \int \sum_{i=1}^n \exp\big(\varepsilon |x_i|\varphi_i(0)\big) \, d\mu^{n,1}(x) \\
=& \sum_{i=1}^n\int_{\mathbb R} \exp\big(\varepsilon |x_i|\varphi_i(0)\big) \, d\mu_i(x_i)\le 2n.
\end{align*}
Therefore
\[ \int \max_i \big(|x_i|\varphi_i(0)\big) \, d\mu^{n,\rho}(x)\le c \log(2n).\]
Consequently, the set
\[ A:=\left\{x\in\mathbb R^n; \; \max_i \frac{|x_i|\varphi_i(0)}{2}\le c \log(2n)   \right\},\]
verifies $\mu^{n,\rho}(A)\ge \frac12$, thanks to Markov's inequality. This implies that the probability
measure
\[\mu^{n,\rho}_{|A}:=\frac{\mu^{n,\rho}(\cdot\cap A)}{\mu^{n,\rho}(A)}= \frac{\mathbf 1_A}{\mu^{n,\rho}(A)}\, \rho\cdot (\mu_1\otimes\cdots
\otimes \mu_n)\]
 obtained by conditioning $\mu^{n,\rho}$ to the set $A$
is close to $\mu^{n,\rho}$ in total variation distance: $$ d_{\mathrm{TV}}(\mu^{n,\rho},\mu^{n,\rho}_{|A} )\le \frac12. $$
Since $A$ is convex and symmetric, we can write $\mu^{n,\rho}_{|A}=\mu^{n,\tilde{\rho}}$ where $\tilde \rho:= \frac{\mathbf 1_A}{\mu^{n,\rho}(A)}\, \rho$ is still log-concave and even.
Since both measures are log-concave, Theorem \ref{th:total-variation} ensures that for some universal constant $\kappa$
\begin{equation}\label{eq:poinc-mu-mutilde}
 C_P(\mu^{n,\rho}) \le \kappa\,  C_P(\mu^{n,\tilde{\rho}}).\end{equation}
We can apply Lemma \ref{lem:poincCM} to $\mu^{n,\tilde{\rho}}$ with the advantage
that this measure is supported on $A$. We obtain, using also Lemma \ref{lem:bound-coef}, that for every odd and locally Lipschitz function $f$,
\begin{align*}
\mathrm{Var}_{\mu^{n,\tilde{\rho}}}(f) & \le \int \sum_i \left( \frac{|x_i|}{2\varphi_i(0)}+\frac{1}{4 \varphi_i(0)^2}\right) (\partial_i f(x))^2 d\mu^{n,\tilde{\rho}}(x)  \\
&\le \max_i \frac{1}{\varphi_i(0)^2}  \int \sum_i \left(|x_i|\frac{\varphi_i(0)}{2}+\frac{1}{4}\right) (\partial_i f(x))^2 d\mu^{n,\tilde{\rho}}(x)  \\
&\le \max_i \frac{1}{\varphi_i(0)^2}  \int_A \left(\max_i \Big(|x_i| \frac{\varphi_i(0)}{2}\Big)+\frac{1}{4}\right) |\nabla f(x)|^2 d\mu^{n,\tilde{\rho}}(x) \\
& \le \max_i \frac{1}{\varphi_i(0)^2}  \left( \frac14+c\log(2n)\right)  \int  |\nabla f|^2 d\mu^{n,\tilde{\rho}}.
\end{align*}
Since $\mu^{n,\tilde{\rho}}$ is log-concave and even, Corollary~\ref{cor:even-sym} ensures that checking the Poincar\'e inequality for odd functions, as we just did, is enough
to conclude that
\[ C_P(\mu^{n,\tilde{\rho}})\le \max_i \frac{1}{\varphi_i(0)^2}  \left( \frac14+c\log(2n)\right). \]
Combining this estimate with \eqref{eq:poinc-mu-mutilde} gives a universal constant $C$ such that
\[ C_P(\mu^{n,\rho}) \le C \log(n) \max_i \frac{1}{\varphi_i(0)^2} \cdot\]
Eventually, for the even log-concave probability measures $\mu_i(dt)=\varphi_i(t)\,dt$ on the real line it is known that $\frac{1}{12} \varphi_i(0)^{-2}\le C_P(\mu_i)\le \varphi_i(0)^{-2}$, see \cite{bobkov}.
\end{proof}

\subsubsection{Examples}
As explained in \cite{ENT-mixtures}, for $p\in (0,2]$ the probability measures on $\mathbb R$ defined by $$ d\nu_p(t)=\exp(-|t|^p)\, dt/Z_p $$ are
Gaussian mixtures. When $p\in [1,2]$ they are in addition log-concave, and Theorem \ref{th:GM-poincare} ensures that for every even log-concave perturbation $\rho$,
 \begin{equation}
 \label{eq:bound-nu-p}  C_P(\nu_p^{n,\rho})\le (1+C \log n)  \, C_p(\nu_p).
 \end{equation}
 for some universal constant $C$. We point out that $\inf_{p\in [1,2]}C_p(\nu_p)>0$ and $\sup_{p\in [1,2]}C_p(\nu_p)<+\infty$, which is easily verified e.g. with the Muckenhoupt criterion \cite{M_P}. This completes the proof of Theorem \ref{th:Exp-poincare} in the case $p=1$.

 \medskip When $p=1$, Theorem \ref{th:Exp-poincare} almost answers the motivating question that we mentioned in the introduction: we unfortunately have a weak dependence in the dimension, but we allow more general perturbations.

\medskip
 When $1<p<2$, using the remark after Lemma~\ref{lem:bound-coef}, we obtain from (\ref{eq_star}) that for the measure $\nu_p$, the coefficients $\alpha_i(t)$ of Lemma~\ref{lem:poincCM} verify
 \[\alpha_i(t)\le c \big(1+|t|^{2-p}\big), \quad t\in \mathbb R,\]
 where $c$ is a universal constant. This improves on Lemma~\ref{lem:bound-coef}, and can be used in the argument of the proof of Theorem~\ref{th:GM-poincare}. Since there exists a universal $\varepsilon>0$ such that for all $p\in (1,2)$, $$ \int \exp\big ( \varepsilon(|t|^{2-p})^{p/(2-p)}\big) d\nu_p(t)\le 2 $$ we arrive by the same method at
 \[ C_P(\nu_p^{n,\rho})\le \big(1+C (\log n)^{\frac{2-p}{p}}\big) \,  C_p(\nu_p).\]
As $\sup_{p\in [1,2]}C_p(\nu_p)<+\infty$, we have proven the following:
\begin{theo}\label{th:GM-poincare2}
Let $1 \leq p \leq 2$. Let $\rho:\mathbb R^n\to \mathbb R^+$ be an even log-concave function such that $d\nu_p^{n,\rho}(x)=\rho(x) \prod_{i=1}^n d\nu_p(x_i)$ is a probability measure on $\mathbb R^n$.
Then
\begin{equation} 
C_P(\nu_p^{n,\rho}) \le (1+C \log  n)^{\frac{2-p}{p}}, \label{eq_1129} \end{equation}
where $C$ is a universal constant.
\end{theo}
Theorem \ref{th:GM-poincare2} implies Theorem \ref{th:Exp-poincare}.
Note that the bound (\ref{eq_1129}) improves on \eqref{eq:bound-nu-p}, and is independent of the dimension for $p=2$ (as expected for log-concave perturbations of the standard Gaussian measure).

\medskip
All the above results deal with even log-concave perturbations of the measures $\nu_p$ and their products $\nu_p^n$, $p\in [1,2]$. The spectral gap of such perturbed measures is controlled uniformly in the perturbation (for any given dimension).  When $p\in[1,2)$ this is not true  for  arbitrary log-concave perturbations (i.e. non necessarily even).  To see this, it is enough to consider the  probability measures $\nu_p$ on $\mathbb R$, and their exponential tilts
\[ d\nu_{p,a}(t)=\frac{1}{Z_{p,a}} e^{-|t|^p+at} dt,\]
where $a$ in an arbitrary real number if $p>1$, and $a\in (-1,1)$ when $p=1$.  Gentil and Roberto \cite{gentil-roberto} have proved that for $p\in [1,2)$,
\[ \sup_a C_P(\nu_{p,a})=+\infty.\]
For $p=2$, the Brascamp-Lieb inequality ensures that the Poincar\'e constant  of any log-concave perturbations of the standard Gaussian measure is dominated by 1.

\subsection{Light tails}
Since  Gaussian mixtures have heavier tails than the Gaussian measure, we now investigate  some measures with lighter tails.

A special and simple case is when  the measures $d\mu_i(t)=e^{-V_i(t)} dt$ have strictly uniformly convex potentials. More specifically, if there exists $\varepsilon>0$ such that for all $i$ and all $t\in \mathbb R$, $V_i"(t)\ge \varepsilon$, then without assuming any symmetry if $\rho$ is log-concave, the probability measure $\mu^{n,\rho}$ also has a potential which is uniformly strictly convex and therefore
\[ C_P\big(\mu^{n,\rho}\big)\le \frac{1}{\varepsilon}\cdot\]

Nevertheless, strict convexity in the large is not sufficient to yield such uniform results. The behaviour of $\mu_i$ around 0 is important as the next examples show:
let $p>2$ and for all $i$, $d\mu_i(t)=\exp(-|t|^p) dt/Z_p$.
For $x\in \mathbb R^n$, let us denote by $\overline{x}=(\sum_i x_i)/n$ its  empirical mean and $Q(x)=\sum_i(x_i-\overline{x})^2/n$ its empirical variance. As a nonnegative quadratic form, $Q$ is convex. Also note that
\[ Q(x)=n \Big|P_{u_n^\bot} x\Big|^2,\]
where $u_n=(1/\sqrt n, \ldots, 1/\sqrt n)\in \mathbb R^n$ is a unit vector on the main diagonal line and $P_{u_n^\bot} $ is the orthogonal projection onto the othogonal complement of this line
 \[u_n^\bot=\big\{x\in \mathbb R^n; \; \sum_i x_i=0\big\}.\]
Let us define $\rho_k:\mathbb R^n\to \mathbb R^+$ as the indicator function of the convex origin-symmetric set $\{x\in \mathbb R^n; Q(x)\le 1/k\}$,  properly normalized so that $\mu^{n,\rho_k}$ is a probability measure (another possible choice would be $\rho_k=\exp(-kQ)/Z_k$).
Then when $k$ tends to $+\infty$ the measure $\mu_{n,\rho_k}$ tends
to the measure obtained by conditioning $\mu_1\otimes\cdots\mu_n=\mu^{n,1}$ to the diagonal line $\mathbb Ru_n$.
With our choice of $d\mu_i(t)=\exp(-|t|^p) dt/Z_p$, this limiting measure is, after isometric identification of  $\mathbb Ru_n$ and $\mathbb R$,
\[  \exp\Big(-\sum_{i=1}^n\Big|\frac{t}{\sqrt n}\Big|^p\Big) \frac{dt}{Z_{n,p}} =\exp\Big(-\Big|\frac{t}{n^{\frac12-\frac1p}}\Big|^p\Big) \frac{dt}{Z_{n,p}} \cdot \]
This measure is the law of $n^{\frac12-\frac1p}Y$ where $Y$ is distributed according to $\mu_1$. Therefore its variance is $n^{1-\frac{2}{p}} \mathrm{Var}(Y)$ and its Poincar\'e constant is $n^{1-\frac{2}{p}} C_P(\mathbb P_Y)$. For $p>2$ this tends to infinity with the dimension.  This growth of the variance in some directions is related to the counterexample in Remark after Theorem \ref{theo:before-rem-cube}, which in a sense corresponds to $p=+\infty$.
The behaviour is very different if we start from Gaussian mixtures, as explained in Theorem \ref{th:GM-odd-functions}.

\begin{rem}
When the functions $V_i$ are strictly uniformly convex in the large, one can obtain Poincar\'e inequalities for small perturbations thanks to a method developped by  Helffer, see e.g.  \cite{helffer}. His approach can be thought of as a variant of the Brascamp-Lieb inequalities where strict convexity is replaced by uniform spectral gap for restrictions to coordinate lines. More precisely, if $d\mu(x)=e^{-V(x)}dx$, consider for $x\in \mathbb R^n$ and $i\in \{1,\ldots,n\}$, the one dimensional probability measure
 \[ d\mu_{|x+\mathbb R e_i} (t) := \frac{1}{Z_{(x_j)_{j\neq i}}}\exp\big(-V(x_1,\ldots, x_{i-1},t,x_{i+1},\ldots, x_n)\big) \, dt,\]
where $(e_i)_{i=1}^n$ is the canonical basis of $\mathbb R^n$.
Then for each $x\in \mathbb R^n$, define the matrix $K(x)$ by
\[ K(x)_{i,j}:=\begin{cases}
   1/C_P(\mu_{|x+\mathbb R e_i}) & \mathrm{when}\, i=j, \\
   \partial^2_{i,j}V(x) & \mathrm{when}\, i\neq j.
\end{cases}\]
If for all $x$, $K(x)$ is positive definite then for all smooth functions $f$, it holds $\mathrm{Var}_\mu(f)\le \int \langle K^{-1}\nabla f,\nabla f\rangle \, d\mu$. In particular, if for all $x$, $K(x)\ge \varepsilon \mathrm{Id}$ then $C_P(\mu)\le \frac{1}{\varepsilon}$.

In our setting of the measures $\mu^{n,\rho}$, the restrictions to coordinate lines are simple (for notational simplicity we present only what happens for $x+\mathbb R e_1$):
 \[ d(\mu^{n,\rho})_{|x+\mathbb R e_1}(t)=e^{-V_1(t)}\rho(t,x_2,\ldots,x_n) \frac{dt}{Z_{(x_j)_{j\ge 2}}} \cdot\]
 If $V_1=U_1+B_1$, where $U_1$ is strictly uniformly convex ($U_1"(t)\ge 1/c_1>0$) and $B_1$ is bounded, then $(\mu^{n,\rho})_{|x+\mathbb R e_1}$ can be viewed as a bounded perturbation (by $B_1$) of  the strictly uniformly convex measure $e^{-U_1}\rho(\cdot,x_2,\ldots,x_n)/\tilde{Z}$ (this is where the log-concavity of $\rho$ is used. Note that no symmetry assumption is needed). It follows from the Brascamp-Lieb inequality and Proposition \ref{prop:bounded-perturbation}  that for all $x$,
 \[ C_P(\mu_{|x+\mathbb R e_1})\le c_1e^{\mathrm{Osc}(B_1)}.\]
 This type of uniform bound allows to get Poincar\'e inequalities for $\mu^{n,\rho}$ provided the non-diagonal terms of the Hessian of $-\log \rho$ are small enough, thanks to Helffer's result. This is especially simple to achieve when $\rho=e^{-Q}$ where $Q$ is a small quadratic form. We refer to Theorem 4.1 in \cite{gentil-roberto} for weaker hypotheses on $B_1$ allowing  similar results.
\end{rem}

\section{Application to convex sets}
Given a non-empty compact convex set $K\subset \mathbb R^d$, we denote by $\lambda_K$ the uniform probability measure
on $K$ (which we may consider in the natural dimension of the affine span of $K$). Also let $B_p^N:=\{x\in \mathbb R^N; \, \|x\|_p\le 1\}$ be the unit ball of $\ell_p^N$.
Recall from Section 2.2 that $C_P(\mu,``linear")=\|\mathrm{Cov}(\mu)\|_{op}$ denotes the smallest constant so that the Poincar\'e inequality is satisfied for all linear functions
with respect to the measure $\mu$.

\begin{theo}\label{theo:ball-section}
Let $n\ge d \ge 2$ and $p\in [1,2]$. Let $E$ be any linear subspace of $\mathbb R^n$ of dimension $d$, then
\[ C_P\big(\lambda_{B_p^n\cap E}\big) \le c  \left( \frac{n}{d}\right)^{\frac{2}{p}-1}\log(n)^{2/p} C_P\big(\lambda_{B_p^n\cap E},``linear"\big),\]
where $c$ is a universal constant.  In particular, if $d\ge n/2$ then for some universal constant $c'$, $C_P\big(\lambda_{B_p^n\cap E}\big) \le c'  \log(d)^2 C_P\big(\lambda_{B_p^n\cap E},``linear"\big)$.
\end{theo}

This result will be deduced from the ones of the previous sections, thanks to a result of Kolesnikov and Milman \cite{kolesnikov-milman}, which allows to transfer Poincar\'e inequalities from log-concave measures to some of their level sets. The next statement is a combination of Theorem 2.5 and Proposition 2.3 in \cite{kolesnikov-milman}.

\begin{theo}[\cite{kolesnikov-milman}] \label{theo:level-set}
Let $d\mu(x)=\exp(-V(x))dx$ be a log-concave probability measure on $\mathbb R^d$, with $\min V=0$.
Then there exists $t>0$ such that the set $K:=\{x\in \mathbb R^d; V(x)\le t\}$ verifies
\begin{enumerate}
\item $C_P(\lambda_K) \le C \cdot C_P(\mu) \cdot \log\big(e+C_P(\mu)\sqrt{d}\big)$,
\item $C_P(\lambda_K,``linear")\ge c>0$,
\end{enumerate}
where $C,c$ are universal constants.
\end{theo}

We shall also need a stability result of the Poincar\'e constant under convergence of measures. For $\vphi: \RR^n \rightarrow \RR$ write $\| \vphi \|_{Lip} = \sup_{x \neq y} |\vphi(x) - \vphi(y)| / |x-y|$ for its Lipschitz seminorm.
According to E. Milman \cite{emanuel}, for any log-concave probability measure $\mu$ on $\RR^n$,
$$ c_1 \sqrt{C_P(\mu)} \leq \sup_{\| \vphi \|_{Lip} \leq 1} \int |\vphi - E_{\mu, \vphi}| d \mu \leq C_2 \sqrt{C_P(\mu)} $$
where $c_1, C_2 > 0$ are universal constants and $E_{\mu, \vphi} = \int \vphi d \mu$.

\begin{proof}[Proof of Theorem~\ref{theo:ball-section}]
For $i=1,\ldots,n$ we set 
 $d\mu_i(t)=d\nu_p(t)=\exp(-|\alpha_p t|^p)dt$, where $\alpha_p=2\Gamma(1+1/p) \in [\sqrt{\pi}, 2]$. These measures are even and log-concave, and their density at 0 is equal to $1$. By Theorem \ref{th:GM-poincare2}, for any even log-concave (and normalized) perturbation $\rho$,
\begin{equation} C_P(\mu^{n,\rho})\le C (\log n)^{\frac{2-p}{p}}
\label{eq_1009} \end{equation} where $C$ is a universal constant. 
Indeed, since the scaling coefficient $\alpha_p$ has the order of magnitude of a universal constant, it may be absorbed in
the universal constant $C$. We apply (\ref{eq_1009}) when $\rho = \rho_{\eps}$ is the normalized indicator of an $\varepsilon$-neighborhood of the subspace $E$. The family of measures $\mu^{n, \rho_{\eps}}$ tends weakly, as $\eps \rightarrow 0$, to the
measure $\tilde\mu$ on the Euclidean space $E$ (that we identify to $\mathbb R^{d}$) with density $\exp(-\|\alpha_p x\|^p_p)/Z_E$, where
\begin{equation}\label{def:ZE}
Z_E = \int_E \prod_{i=1}^n \exp(-|\alpha_p x_i|^p)d^Ex
\end{equation}
 is the integral over $E$ of the density of $(\nu_p)^n$. We claim that
 \begin{equation}
C_P(\tilde{\mu})\le 2 C (\log n)^{\frac{2-p}{p}}. \label{eq_1011}
 \end{equation}
 Indeed, otherwise there exists a smooth $\vphi: E \rightarrow \RR$
 with $Var_{\tilde{\mu}}(\vphi) > 2 C \log(n) \cdot \int |\nabla \vphi|^2 d \tilde{\mu}$. By multiplying $\vphi$ with a slowly-varying cutoff function, we may assume that $\vphi$ is compactly-supported in $E$
 (the argument is standard, see Section \ref{sec_app1} below for details).
 We set $f(x) = \vphi( P_E x)$, where $P_E$ is the orthogonal projection onto $E$ in $\RR^n$. Then as $\eps \rightarrow 0^+$,
 $$ Var_{\mu^{n, \rho_{\eps}}}(f) \longrightarrow Var_{\tilde{\mu}} (\vphi) \quad \text{and} \quad \int |\nabla f|^2 d \mu^{n, \rho_{\eps}}
 \longrightarrow \int |\nabla \vphi|^2 d \tilde{\mu},
 $$
 in contradiction to (\ref{eq_1009}). This completes the proof of (\ref{eq_1011}).
  In order to apply Theorem \ref{theo:level-set}, we need to rescale $\tilde\mu$. Let $Y$ a random vector on $E$ with law $\tilde\mu$, then for $\lambda >0$ the random vector $\lambda Y$ has a distribution of density on $E$ given by
 \[ \exp\left( -\left\| \frac{\alpha_px}{\lambda}\right\|_p^p-\log(Z_E)-d\log(\lambda)\right).\]
This suggests to set  $\lambda_E:= Z_E^{-1/d}$.
For this choice,
the probability measure $\mu(dx)=\exp( -\|\alpha_p x/\lambda_E\|_p^p) d^E_x$ on $E$ verifies
\[ C_P(\mu)=\lambda_E^2 C_P(\tilde{\mu}) \le \lambda_E^2 C (\log n)^{\frac{2-p}{p}} =
C Z_E^{-\frac2d} (\log n)^{\frac{2-p}{p}}.
\]
%
In order to bound the latter quantity from above, we need a lower bound for $Z_E$, as defined in \eqref{def:ZE}. This can be done by general results on sections of isotropic measures. More precise bounds were obtained by Meyer and Pajor \cite{meyer-pajor} in their investigation of  extremal volumes of sections of $B_p^n$ (they observe that $Z_E=\mathrm{Vol}_d(B_p^n\cap E)/ \mathrm{Vol}_d(B_p^d) $). For our purpose, a simple bound based  on the inradius of $B_n^p$ is the most effective:  since $p\le 2$, for any $x\in \mathbb R^n$, $\|x\|_p \le n^{\frac1p-\frac12} \|x\|_2$, thus
\begin{align*}
Z_E & =\int_E \exp\big(-\|\alpha_p x\|_p^p\big) \, d^Ex
\ge \int_E \exp\big(-\| n^{\frac1p-\frac12} \alpha_p x\|_2^p\big)\,  d^Ex.
\end{align*}
The later integral takes the same value for all $d$-dimensional vector spaces $E$. Therefore
\begin{align*}
Z_E & \ge \int_{\mathbb R^d} \exp\big(-\| n^{\frac1p-\frac12} \alpha_p x\|_2^p\big)\,  dx \\
 &= \mathrm{Vol}_d(B_2^d) \int_0^{+\infty}d r^{d-1} \exp\big( -(n^{\frac1p-\frac12} \alpha_p r)^p \big) dr \\
 &= \mathrm{Vol}_d(B_2^d) \frac{\int_0^{+\infty}d s^{d-1} \exp\big( -s^p \big) ds}{\big(n^{\frac1p-\frac12} \alpha_p\big)^d} = \left(\frac{\sqrt\pi}{n^{\frac1p-\frac12} \alpha_p} \right)^d \frac{\Gamma\big(1+\frac{d}{p}\big)}{\Gamma\big(1+\frac{d}{2}\big)}.
 \end{align*}
For $x$ large, $\Gamma(1+x)^{\frac1x}\sim x/e$, we get that for some numerical constants $c,c'$,
\[ Z_E^{-\frac2d}\le c \frac{\alpha_p^2 n^{\frac2p-1}}{\pi} \frac{\frac{d}{2e}}{(\frac{d}{pe})^{\frac{2}{p}}} \le c' \left(\frac{n}{d} \right)^{\frac2p-1}.
\]
This leads to
\[ C_P(\mu) \le
C' \left(\frac{n}{d} \right)^{\frac2p-1} (\log n)^{\frac{2-p}{p}}.
\]
Applying Theorem~\ref{theo:level-set} to $\mu$ provides $t>0$ so that the set $K_E:= \{x\in E; \|\alpha_p x/\lambda_E\|^p_p\le t\}=\alpha_p^{-1}\lambda_E t^{\frac1p} \big(B_p^n\cap E\big)$ verifies
\begin{align*}
C_P(\lambda_{K_E}) &\le C C'\left(\frac{n}{d} \right)^{\frac2p-1} (\log n)^{\frac2p-1} \log\left( e+C'\left(\frac{n}{d} \right)^{\frac2p-1} (\log n)^{\frac2p-1} \sqrt{d}\right)  C_P(\lambda_{K_E},``linear") \\
& \le C'' \left(\frac{n}{d} \right)^{\frac2p-1}\log(n)^{\frac2p} C_P(\lambda_{K_E},``linear").
 \end{align*}
Since the constants $C_P(\cdot)$ and $C_P(\cdot, ``linear)$ are both 2-homogeneous with respect to dilations of the underlying measure, we get the claim
\[ C_P(\lambda_{B_p^n\cap E})   \le C'' \left(\frac{n}{d} \right)^{\frac2p-1}\log(n)^{\frac2p} C_P(\lambda_{B_p^n\cap E},``linear").\]
\end{proof}

Corollary \ref{cor1} of the introduction clearly follows from
Theorem~\ref{theo:ball-section}.

\section{Appendix: approximation results}

\subsection{Density of test functions}
\label{sec_app1}

Let $\mu$ be a log-concave measure on $\RR^n$.
We assume that the support of $\mu$ is not contained in an affine subspace of lower dimension, as otherwise, we may just work in the lower dimensional subspace. Hence $\mu$ is of the form $\rho(x) dx$
where $\rho$ is a log-concave function.  Let $\Omega$ be the interior of the support of $\mu$. It is convex and non-empty (assuming that $\mu$ is not the zero measure). The function $\rho$ is positive on $\Omega$ and vanishes outside of $\Omega$.
Write $C_c^{\infty}(\Omega)$ for the space of smooth functions, compactly-supported in $\Omega$. By definition, $H^1(\Omega, \mu)=H^1(\mu)$ is the set of (equivalence classes of) functions $f$ in $L^2(\mu)$, for which there exist functions $g_i\in L^2(\mu)$ such that for all $1\le i\le n$ and for all $\varphi \in C_c^{\infty}(\Omega)$,
\[ \int_\Omega \partial_i\varphi(x) f(x) \, dx=-\int_\Omega \varphi(x) g_i(x) \, dx.\]
Classically, $g_i$ is called a weak partial derivative of $f$ (viewed as a function on $\Omega$). The weak gradient $(g_i)_i$ is simply denoted by $\nabla f$ and
\begin{equation}   \| f \|_{H^1(\mu)} = \sqrt{  \int_{\RR^n} f^2 d \mu+\int_{\RR^n} |\nabla f|^2 d \mu}. \label{eq_H1} \end{equation}

The following basic result will be useful:
\begin{proposition}\label{prop:density}
Let $\mu$ be a log-concave measure on $\RR^n$. Then the set $C_c^{\infty}(\RR^n)$ is dense in $H^1(\mu)$.
\end{proposition}
Several textbooks are dedicated to the study of density of smooth functions in weighted Sobolev spaces (see e.g Kufner \cite{KUFNER}),
and they  consider more difficult situations. Nevertheless, we found it hard to spot a reasonably self-contained justification of the above proposition. This is why we include an ad-hoc proof, which relies only on very basic facts about density of smooth functions in $H^1_{\mathrm{loc}}(\Omega, dx)$
(see e.g. \cite[Chapter 5]{EVANSbook}). Local approximation in any compact subset of $\Omega$ is easy, since on such a subset  $\rho$ is upper bounded, and bounded away from $0$, hence the result for the Lebesgue measure applies. To derive approximation up to the boundary, one usually approximates $f$ by functions which are defined somewhat outside of $\Omega$, on which local approximation applies up to the boundary. To build such functions, when the boundary of $\Omega$ is regular enough, one usually proceeds by local translations of $f$. In our case,
since $\Omega$ is convex, a single global dilation does the job.

\begin{proof}[Proof of Proposition \ref{prop:density}]
Let us set some more notation. Our problem is invariant by translation. Hence we may assume that the origin $0\in \Omega$. The latter being open, there exists $r>0$ such that $B(0,r)\subset \Omega$.
Let $f$ be an arbitrary function in $H^1(\mu)$. Our goal is to build compactly-supported smooth functions which are arbitrarily close to $f$
in the $H^1(\mu)$ norm.

We first reduce matters to functions $f$ with compact support in $\RR^n$. Indeed, given a general $f\in H^1(\mu)$, consider a bump function $\theta:\RR^n\to [0,1]$ which is infinitely differentiable and such that $\theta(x)=1$ if $x\in B(0,1)$, while $\theta(x)=0$ if $x\not\in B(0,2)$.
For any integer $n\ge 1$, define
\[ f_{|n}(x):=\theta(x/n)f(x),\qquad x\in \RR^n.\]
It is supported in $B(0,2n)$ and belongs to $L^2(\mu)$ since $|f_{|n}|\le |f|$. By dominated convergence
\[  \|f-f_{|n} \|_{L^2(\mu)}^2=\int f(x)^2 (1-\theta(x/n))^2 d\mu(x)\]
tends to 0 when $n\to +\infty$. Since $\partial_i f_{|n}=
\theta(\cdot/n) \partial_if +\frac1n \partial_i\theta (\cdot/n)f$,
\[  \|\partial_if-\partial_if_{|n} \|_{L^2(\mu)}\le \frac1n \|f \partial_i\theta(\cdot/n) \|_{L^2(\mu)}+ \|\partial_if-(\partial_if)_{|n} \|_{L^2(\mu)}^2  \]
also tends to 0 when $n$ grows. Indeed, the functions $\partial_i\theta$ are uniformly bounded, and we may apply the latter convergence of truncated functions to $\partial_i f\in L^2(\mu)$. 

\begin{lemma}\label{lem:density-continuous}
The set $C_c^{\infty}(\Omega)$ is dense in $L^2(\mu)$.
\end{lemma}
\begin{proof}
By the above truncation argument, it is enough to approximate functions with compact support in $\RR^n$. Let $h\in L^2(\mu)$ be with support in the open ball $B(0,R)$ for some $R$. By dominated convergence,
\[ \lim_{\varepsilon\to 0^+} \int (f \mathbf 1_{(1-\varepsilon) \Omega}-f)^2 d\mu=0.\]
For $\varepsilon>0$, the set $\widetilde{\Omega}:=(1-\varepsilon)\Omega\cap B(0,R)$ is relatively compact in $\Omega$, hence there exists $c>0$ such that
$c\le \rho(x)\le \frac1c$ for all $x\in \widetilde{\Omega}$. Hence $
f \mathbf 1_{(1-\varepsilon) \Omega}\in L^2(\mu)$ also belongs to
the unweighted Lebesgue space $L^2(\widetilde{\Omega}, dx)$, in which it is classical that compactly supported smooth functions are dense.
Therefore there is a sequence $g_n\in C_c^{\infty}(\widetilde{\Omega})$ which converges to $f \mathbf 1_{(1-\varepsilon) \Omega}$ for the $L^2(\widetilde{\Omega}, dx)$-topology. Since $\rho\le \frac1c$ on $\widetilde{\Omega}$, and all functions are supported in $\widetilde{\Omega}$ the convergence also holds in the topology of $L^2(\mu)$.
\end{proof}

For $f\in L^2(\mu)$ and a parameter $\delta\in(0,\frac12)$ we introduce the dilated function $f_\delta$ defined by
\[ f_\delta(x):=f\big((1-\delta)x \big), \quad x\in \frac{1}{1-\delta}\Omega \supset \Omega.\]
These functions are defined outside of $\Omega$ but provide a fair approximation of $f$ for small $\delta$:
\begin{lemma}\label{lem:dilation}
Let $f\in L^2(\mu)$, with bounded support. Then for all $\delta\in(0,\frac12)$, $f_\delta\in L^2(\mu)$ and
when $\delta$ tends to 0, $f_\delta$ converges to $f$ in the topology of $L^2(\mu).$

If in addition, $f\in H^1(\mu)$ then the convergence holds in the topology of $H^1(\mu)$.
\end{lemma}

\begin{proof}[Proof of Lemma \ref{lem:dilation}]
Assume that $f$ is supported in $B(0,R)$. Let us compute the squared $L^2$ norm of $f_\delta$:
\[ \int_\Omega f\big((1-\delta)x\big)^2 \rho(x) dx= (1-\delta)^{-n} \int_{(1-\delta)\Omega} f(y)^2\rho\Big( \frac{y}{1-\delta}\Big) dy.\]
The log-concavity of $\rho$ yields $\rho(y)\ge \rho\Big( \frac{y}{1-\delta}\Big)^{1-\delta} \rho(0)^\delta$. Rearranging gives
\[ \rho\Big( \frac{y}{1-\delta}\Big) \le \rho(y) \left( \frac{\rho(y)}{\rho(0)}\right)^{\frac{\delta}{1-\delta
}}.\]
Since $\rho$ is upper-bounded on the compact support of $f$
(see e.g., \cite[Lemma 2.2.1]{BOOKgreek}), there exists a constant $C_R$ such that for all $y$, $f(y)^2 \rho\Big( \frac{y}{1-\delta}\Big) \le C_R f(y)^2 \rho(y)$. Thus $\| f_\delta\|_{L^2(\mu)}\le 2^n C_R \| f\|_{L^2(\mu)}$.

\smallskip
For any $\varepsilon >0$, Lemma \ref{lem:density-continuous} provides $g\in C_c^{\infty}(\Omega)$ (supported also inside $B(0,R)$ as the proof of the lemma shows) such that $\|f-g\|_{L^2(\mu)}\le \varepsilon$. Then
\[ \| f-f_\delta\|_{L^2(\mu)} \le \| f-g\|_{L^2(\mu)} +\| g-g_\delta\|_{L^2(\mu)} +\| g_\delta-f_\delta\|_{L^2(\mu)}.\]
By the above norm estimate $ \| g_\delta-f_\delta\|_{L^2(\mu)}\le 2^n C_R\| g-f\|_{L^2(\mu)}$. Moreover since $g$ is uniformly continuous,
and $g_\delta$ as well as $g$ vanish outside of $B(0,2R)$,
\[ \| g_\delta-g\|_{L^2(\mu)}^2=\int_{B(0,2R)} \big|g(x)-g((1-\delta)x)\big|^2d\mu(x) \le  \mu(B(0,2R)) \omega_g(2R\delta)^2,\]
where $\omega_g$ denotes the modulus of continuity of $g$.
Combining the above estimates gives
\[\limsup_{\delta\to 0^+}  \| f-f_\delta\|_{L^2(\mu)} \le (1+2^n C_R)\varepsilon,\]
for every $\varepsilon>0$. This proves the convergence of $f_\delta$ to $f$.

\smallskip
Eventually, if $f\in H^1(\mu)$, observe that
\[ \| \partial_if-\partial_i(f_\delta)\|_{L^2(\mu)}= \| \partial_if-(1-\delta)(\partial_if)_\delta\|_{L^2(\mu)}\le \delta \|\partial_i f\|_{L^2(\mu)}+(1-\delta)  \| \partial_if-(\partial_i f)_\delta\|_{L^2(\mu)}\]
tends to 0 when $\delta$ does, by  the result that we just proved, applied to $\partial_i f \in L^2(\mu)$.
\end{proof}

We are now ready to complete the proof Proposition \ref{prop:density}.
As already explained, it is enough to approximate an arbitrary $f\in H^1(\mu)$ whose  support is contained in $B(0,R)$ for some $R$.
For $\delta \in (0,\frac12)$, we consider the dilated function $f_\delta$ defined on $(1-\delta)^{-1}\Omega$. The last ingredient is regularization by convolution: let $\eta:\RR^n\to \RR^+$ be a standard mollifier, meaning $\eta$ is of class $C^\infty$, $\eta(x)=0$ if $|x|\ge 1$ and $\int \eta(x) dx=1$. For $\varepsilon\in(0,1)$, consider $\eta^\varepsilon$ defined for $x\in \RR^n$
by
\[ \eta^\varepsilon(x)=\varepsilon^{-n} \eta\Big( \frac{x}{\varepsilon}\Big),\]
and the convolution $f_\delta\ast \eta^\varepsilon$.
Observe that $f_\delta\in H^1_{\mathrm{loc}}\big((1-\delta)^{-1}\Omega, dx \big)$. Indeed for any compact $K\subset (1-\delta)^{-1}\Omega$,
\[\int_K f_\delta(x)^2 dx=(1-\delta)^{-n} \int_{(1-\delta)K}f(x)^2dx \le C_K \int_{(1-\delta)K}f^2 \rho \le C_K\int f^2 d\mu<+\infty,\]
where we have used that $\rho$ attains a positive minimum on the compact set $(1-\delta)K\subset \Omega$. The same argument applies to the partial derivatives of $f$. Thus, according to \cite[Theorem 1 of Section 5.3]{EVANSbook}, $f_\delta\ast \eta^\varepsilon$ is well defined and infinitely differentiable on the set
\[ U_\varepsilon:=\Big\{ x\in (1-\delta)^{-1}\Omega; \; \mathrm{dist}\big(x, \big((1-\delta)^{-1}\Omega\big)^c\big)>\varepsilon\Big\}.\]
Moreover when $\varepsilon$ tends to 0, $f_\delta\ast\eta^\varepsilon$
tends to $f_\delta$ in $H^1_{\mathrm{loc}}\big((1-\delta)^{-1}\Omega, dx \big)$.

As $\Omega\cap B(0, 2R+1)\subset\subset (1-\delta)^{-1}\Omega$,
we can deduce that  when $\varepsilon$ tends to 0, $f_\delta\ast\eta^\varepsilon$
tends to $f_\delta$ in $H^1\big(\Omega\cap B(0,2R+1), dx \big)$.
Taking into account the fact that $f_\delta$ and $f_\delta\ast\eta^\varepsilon$ vanish outside of $B(0,2R+1)$ and that the log-concave
function $\rho$ is bounded from above in $\Omega \cap B(0,2R+1)$,
we can conclude that $\lim_{\varepsilon\to 0^+} \| f_\delta\ast\eta^\varepsilon-f_\delta\|_{H_1(\mu)}=0$.

To approximate the orginal function $f$ up to accuracy $\alpha>0$, we simply write
\[ \|f_\delta\ast \eta^\varepsilon-f \|_{H^1(\mu)}\le \| f_\delta\ast \eta^\varepsilon-f_\delta\|_{H^1(\mu)}+ \|f_\delta-f \|_{H^1(\mu) },\]
use Lemma~\ref{lem:dilation} to find a $\delta$ for which the last term is at most $\alpha/2$. Then we let $\varepsilon$ tend to zero.

Since $B(0,r)\subset \Omega$, the set $U_\varepsilon$ contains $((1-\delta)^{-1}-\frac{\varepsilon}{r})\Omega$ when $\varepsilon <r(1-\delta)^{-1}$. Consequently, if $\varepsilon<\delta^2(r(1-\delta))^{-1}$
then $(1+\delta)\Omega \subset  U_\varepsilon$. So the above approximations of $f$ are $C^\infty$ on a larger set than $\Omega$.
Since they also vanish outside of $B(0,2R+1)$, we may modify them outside of $\Omega$  in order to obtain functions in $C_c^\infty(\RR^n)$.
\end{proof}

\subsection{Proof of Lemma \ref{lem_1111}}

This section is devoted to the proof of Lemma \ref{lem_1111}.
We may assume that the support of $\mu$ is not contained in an affine subspace of lower dimension, as otherwise, we may just work in the lower dimensional subspace.
Proposition \ref{prop_1004} is proven above under the additional assumption that $\mu$ has a smooth density that is  positive everywhere  in $\RR^n$. Our goal here is to prove the inequality
\begin{equation}
\mathrm{Var}_{\mu}(f) \leq \sum_{i=1}^n \| \partial_i f \|^2_{H^{-1}(\mu)}
\label{eq_431} \end{equation}
 in the case of a general, log-concave, finite measure $\mu$ in $\RR^n$, and a general function $f \in L^2(\mu)$ whose weak partial derivatives $\partial^1 f, \ldots, \partial^n f$ belong to $L^2(\mu)$ and satisfy $\int \partial^i f d \mu = 0$. Recall the definition
 (\ref{eq_H1}) of the $H^1(\mu)$-norm, and that $H^1(\mu)$ is the space of $f \in L^2(\mu)$ with $\| f \|_{H^1(\mu)} < \infty$.
Recall from Proposition \ref{prop:density} that  the collection of all smooth, bounded, Lipschitz functions $u: \RR^n \rightarrow \RR$ is dense in $H^1(\mu)$.

\medskip
Next, we claim that both the left-hand side and the right-hand side of (\ref{eq_431}) depend continuously on the function  $f$
with respect to the $H^1(\mu)$-topology, as long as we keep the constraint $\int \partial_i f d \mu = 0$ for all $i$. Indeed, 
the $H^1(\mu)$-norm is stronger than the $L^2(\mu)$-norm, and hence  $\mathrm{Var}_{\mu}(f)$ is continuous in $f$ with respect to the $H^1(\mu)$-norm.
As for the right-hand side of (\ref{eq_431}), by inequality (\ref{eq_915}) above,
$$ \| \partial_i f - \partial_i \tilde{f} \|_{H^{-1}(\mu)} \leq C_p(\mu) \|
\partial_i f - \partial_i \tilde{f} \|_{L^2(\mu)} \leq C_p(\mu) \| f - \tilde{f} \|_{H^1(\mu)}. $$
It therefore suffices to prove (\ref{eq_431}) under the additional assumption that $f$ is a smooth function, bounded in $\RR^n$ together with its first partial derivatives, such that $\int f d \mu = 0$ and also $\int \partial_i f d \mu = 0$ for $i=1,\ldots,n$.

\begin{lemma} \label{lem:approx-mu-above}
	Let $\mu$ be a finite measure on $\RR^n$
	whose density $\rho$ is log-concave. Then
	there exists a sequence of functions $(\rho_k)_{k \geq 1}$  with the following properties:
	\begin{enumerate}
		\item[(i)] For any $k$, the function $\rho_k: \RR^n \rightarrow (0, \infty)$ is a
		smooth, everywhere-positive, integrable, log-concave function on $\RR^n$ such that $\rho \leq \rho_k$ pointwise.
		\item[(ii)] Write $S \subseteq \RR^n$ for the interior of the support of $\mu$, which is an open, convex set of a full $\mu$-measure. Then $\rho_k \longrightarrow \rho$ locally uniformly in $S$.	
		\item[(iii)]
		For any measurable function $\vphi: \RR^n \rightarrow \RR$ that grows at most polynomially at infinity,
		$$ \int_{\RR^n} \vphi \rho_k \stackrel{k \rightarrow \infty} \longrightarrow \int_{\RR^n} \vphi  \rho. $$
	\end{enumerate}
\end{lemma}

Lemma~\ref{lem:approx-mu-above} will be proven shortly.
We apply the lemma  to $\mu$ and denote by $\mu_k$ the measure whose density is $\rho_k$. Let $\theta_k \in \RR^n$ and $\alpha_k \in \RR$ be such that $\tilde{f}_k(x) = f(x) + \langle \theta_k, x \rangle + \alpha_k$ satisfies
$$ \int_{\RR^n} \tilde{f}_k d \mu_k = 0,
\qquad \text{and} \qquad \int_{\RR^n} \partial_i \tilde{f}_k d \mu_k = 0 \quad (i=1,\ldots,n).$$
 We deduce from  Item (iii) of Lemma~\ref{lem:approx-mu-above} that $\theta_k$
and $\alpha_k$ tend to zero as $k \rightarrow \infty$. It also follows that
$$ \mathrm{Var}_{\mu_k}(\tilde{f}_k) \stackrel{k \rightarrow \infty}\longrightarrow \mathrm{Var}_{\mu}(f). $$
All that remains in order to complete the proof of Lemma \ref{lem_1111} is to prove that for $i=1,\ldots,n$ and $g = \partial^i f$,
\begin{equation}
\limsup_{k \rightarrow \infty} \| g - E_k(g) \|_{H^{-1}(\mu_k)} \leq \| g - E(g) \|_{H^{-1}(\mu)}. \label{eq_949}
\end{equation}
where $E_k(g) = \int_{\RR^n} g d \mu_k / \mu_k(\RR^n)$ and $E(g) = \int g d \mu / \mu(\RR^n)$.
We will actually prove (\ref{eq_949}) for any bounded function $g: \RR^n \rightarrow \RR$.
 Normalizing, we may assume that $\sup |g| \leq 1$.
Let $\eps > 0$. It suffices to prove that
\begin{equation}
\limsup_{k \rightarrow \infty} \| g - E_k(g) \|_{H^{-1}(\mu_k)} \leq \| g - E(g) \|_{H^{-1}(\mu)} + 2 \eps \cdot \left[ C_P(\mu) + \sup_k C_P(\mu_k) \right]. \label{eq_503}
\end{equation}
Indeed, $\sup_k C_P(\mu_k) < \infty$ (see  \cite{bobkov}). Let
$T \subset S$ be a compact, convex set with
$$ \mu(\RR^n \setminus T) < \eps^2 / 4. $$
Then there exists $k_0$ such that
$\mu_k(\RR^n \setminus T) < \eps^2 / 4$ for all $k \geq k_0$. Define
$ h = g \cdot 1_T $
where $1_T$ is the characteristic function of $T$, which equals one in $T$ and vanishes elsewhere. Then for all $k > k_0$,
$$ \| g - E(g) - h + E(h)  \|_{L^2(\mu)} < \eps \quad \textrm{and also} \quad \| g - E_k(g) - h + E_k(h) \|_{L^2(\mu_k)} < \eps. $$ In view of (\ref{eq_915}) above,
we see that  (\ref{eq_503}) would follow once we prove that
\begin{equation}
\limsup_{k \rightarrow \infty} \| h - E_k(h) \|_{H^{-1}(\mu_k)} \leq \| h - E(h) \|_{H^{-1}(\mu)}. \label{eq_506}
\end{equation}
However, $h$ is supported in the compact set $T \subset S$, where $S$ is an open set in which $\rho$ is positive. The convergence of $\rho_k$ to the density $\rho$ is uniform in $T$. For  $k \geq 1$ let
$u_k: \RR^n \rightarrow \RR$ be a locally-Lipschitz function in $L^2(\mu_k)$ with
$\int_{\RR^n} |\nabla u_k|^2 d \mu_k \leq 1$ and $\int u_k d \mu_k = 0$ and
$$ \| h - E_k(h) \|_{H^{-1}(\mu_k)} \leq \frac{1}{k} + \int_{\RR^n} h u_k d \mu_k. $$
Since $\rho_k \geq \rho$, necessarily $\int |\nabla u_k|^2 d \mu \leq 1$. Therefore,
\begin{align} \nonumber \| h - E(h) \|_{H^{-1}(\mu)} & \geq \int_{\RR^n} h u_k  d \mu = \int_{T} h u_k d \mu_k
+ \int_T h u_k (\rho - \rho_k)
\\ & \geq \| h - E_k(h) \|_{H^{-1}(\mu_k)} - \frac{1}{k}
- \frac{\sup_T |\rho_k - \rho|}{\inf_T \rho_k} \cdot \int_T |u_k| d \mu_k.
\label{eq_547}
\end{align}
Note that $\sup_T |\rho_k - \rho|$ tends to zero with $k$, while $\inf_T \rho_k$ is bounded away from zero for a sufficiently large $k$.
Moreover, $\left( \int_T |u_k| d \mu_k \right)^2 \leq \mu_k(\RR^n)\int_{\RR^n} u_k^2 d \mu_k \leq \sup_k  \mu_k(\RR^n)C_P(\mu_k) < \infty$. By letting $k$ tend to infinity, we thus obtain (\ref{eq_506}) from (\ref{eq_547}).
This  completes the proof of Lemma \ref{lem_1111}.

\begin{proof}[Proof of Lemma \ref{lem:approx-mu-above}] Set $\psi(x) = -\log \rho(x)$ for $x \in S$ and $\psi(x) = +\infty$ for $x \not \in S$. The function $\psi$ is convex in $\RR^n$,
and the integrability of $e^{-\psi}$ implies that there exists
$A \in (0,1)$ and $B > 0$ such that \begin{equation} \psi(x) \geq A |x| - B \qquad \text{for all} \ x \in \RR^n. \label{eq_1709} \end{equation}
See, e.g., \cite[Lemma 2.2.1]{BOOKgreek}  for a quick proof. For $k \geq 1$ denote
	\begin{equation} \tilde{\psi}_k(x) = \inf_{y \in S} \left[ \psi(y) + k |x - y| \right] \qquad \qquad \text{for} \ x \in \RR^n.
	\label{eq_1659}
	\end{equation}
	The function $\tilde{\psi}_k$ is a Lipschitz function in $\RR^n$, being the infimum of a family of $k$-Lipschitz functions. It is also convex,
	since it is the infimum-convolution of two convex functions (see, e.g., Rockafellar \cite[Section 5]{roc}).
	 Clearly $\tilde{\psi}_k \leq \psi$. From  (\ref{eq_1709}) and (\ref{eq_1659}), for any $k \geq 1$ and $x \in \RR^n$,
	\begin{equation} \tilde{\psi}_k(x) \geq \inf_{y \in S} [A |y| + k|x-y| - B] \geq
	\inf_{y \in S} [A |y| + A|x-y| - B]
	\geq A |x| - B. \label{eq_421} \end{equation}
	Fix a smooth probability density $\theta: \RR^n \rightarrow \RR$ supported in the unit ball $B(0,1)$. Write $\theta_{\eps}(x) = \eps^{-n} \theta(x / \eps)$ and define
	$$ \psi_k = \tilde{\psi}_k * \theta_{1/k^2} - 1 / k. $$
	The function $\psi_k$ is still $k$-Lipschitz and convex, since a convolution preserves this properties.
	 We claim that
	\begin{equation}
	\tilde{\psi}_k - 1/k \leq \psi_k \leq \tilde{\psi}_k \leq \psi \qquad \qquad \textrm{pointwise in} \ \RR^n. \label{eq_414}
	\end{equation}
	Indeed, since $\tilde{\psi}_k$ is convex and $\theta_{1/k^2}$ is a probability density, by Jensen's inequality,
	$$ \psi_k + 1 / k = \tilde{\psi}_k * \theta_{1/k^2} \geq \tilde{\psi}_k,  $$
	which implies the left-hand side inequality in (\ref{eq_414}). On the other hand, since $\tilde{\psi}_k$ is $k$-Lipschitz
	and $\theta_{1/k^2}$ is supported in the ball of radius $1/k^2$ centered at the origin in $\RR^n$,
	$$ \psi_k + 1/k = \tilde{\psi}_k * \theta_{1/k^2} \leq \tilde{\psi_k} + k / k^2 = \tilde{\psi}_k + 1/k, $$
	implying the inequality in the middle in (\ref{eq_414}). This completes the proof of (\ref{eq_414}), as we have already seen the right-hand side inequality in (\ref{eq_414}).
	
	\medskip Let us now set $\rho_k = \exp(-\psi_k)$. Since $\psi_k$ is a smooth, convex, Lipschitz function,
	the function $\rho_k$ is smooth, everywhere-positive and log-concave. It satisfies $\rho_k \geq \rho$ thanks to (\ref{eq_414}). The integrability of $\rho_k$ follows from (\ref{eq_421}) and (\ref{eq_414}), completing the proof of (i).
	
	\medskip The function $\psi$ is locally-Lipschitz in $S$ since it is convex. It thus follows from (\ref{eq_1659}) that
	$\tilde{\psi}_k$ tends to $\psi$ pointwise in $S$, as $k \rightarrow \infty$. According to \cite[Theorem 10.8]{roc}, the convergence is locally-uniform in $S$.
	Since $\tilde{\psi}_k$ tends to $\psi$ locally uniformly in $S$, we learn from (\ref{eq_414}) that also $\psi_k$ tends to $\psi$ locally uniformly in $S$. Consequently, $\rho_k \longrightarrow \rho$ locally uniformly in $S$, as stated in (ii). It remains to prove (iii). From
	(\ref{eq_1709}), (\ref{eq_421}) and (\ref{eq_414}),
	$$ \rho_k(x) \leq e^{B+1 - A |x|} \qquad \textrm{for all} \ k \geq 1, x \in \RR^n. $$
	Hence the function $|\vphi(x)| e^{B+1 - A |x|}$ is an integrable majorant for the sequence of functions $(\vphi \rho_k)_{k \geq 1}$ in $\RR^n$. In view of Lebesgue's dominated convergence theorem, all that remains in order to prove (iii) is to show that $\rho_k \longrightarrow \rho$ almost everywhere in $\RR^n$. We already know that $\rho_k \longrightarrow \rho$ in $S$. Since $S$ is a convex set, its boundary has a zero Lebesgue measure. Thus, it suffices to fix a point $x \in \RR^n$ which is not in the closure of $S$, and prove that
	\begin{equation} \rho_k(x) \stackrel{k \rightarrow \infty} \longrightarrow 0. \label{eq_458} \end{equation}
	There exists $\eps > 0$ such that the ball $B(x,\eps)$ is disjoint from $S$.
	It follows from (\ref{eq_1709}) and (\ref{eq_1659}) that $\tilde{\psi}_k(x) \geq k \eps - B$ for all $k$.
	From (\ref{eq_414}) we thus learn that
	$\psi_k(x) \geq k \eps - B - 1/k \longrightarrow \infty$ as $k \rightarrow \infty$. This implies (\ref{eq_458}), completing the proof of the lemma.	
\end{proof}


\bigskip
\noindent Institut de Math\'eatiques de Toulouse, CNRS UMR 5219, Universit\'e Paul Sabatier,
31062 Toulouse Cedex 09, France.

\hfill \verb"barthe@math.univ-toulouse.fr"

\bigskip
\noindent Department of Mathematics, Weizmann Institute of Science, Rehovot 76100 Israel.

\smallbreak
\hfill \verb"boaz.klartag@weizmann.ac.il"

\end{document}